\documentclass[11pt]{amsart}
\usepackage{amssymb}
\usepackage{amsfonts} 
\usepackage{amsmath}                           
\usepackage{amsthm}
\usepackage{cancel} 
\usepackage{hyperref}
\usepackage[verbose]{geometry}

\usepackage[english]{babel}
\usepackage[latin1]{inputenc}
\usepackage{amsmath,amssymb,amsthm,epsfig}
\usepackage{eufrak}
\usepackage{color}
\usepackage{mathrsfs}
\usepackage{latexsym}

\theoremstyle{plain}
   
\newtheorem{teo}{Theorem}[section]
\newtheorem{conj}{Conjecture}[section]
\newtheorem{prop}{Proposition}[section]
\newtheorem{defi}{Definition}[section]

\newtheorem{cor}{Corollary}[section]
\theoremstyle{definition}
\newtheorem{exe}{Example}[section]
\newtheorem{obs}{Remark}[section]
\begin{document}
\title[Codimension one distributions on compact toric orbifolds]
{On codimension one holomorphic distributions on compact toric orbifolds}

\author{Miguel Rodr\'iguez Pe\~na}
\address{Arnulfo Miguel Rodr\'iguez Pe\~na\\
Universidade Federal de S\~ao Jo\~ao del-Rei - UFSJ \\
Departamento de Estat\'istica, F\'isica e Matem\'atica - DEFIM \\
Ouro Branco MG, Brazil.}
\email{miguel.rodriguez.mat@ufsj.edu.br}

\date\today

\begin{abstract}
The number of singularities, counted with multiplicity, of a generic codimension one holomorphic distribution on a compact toric orbifold is determined. As a consequence, we provide a classification for regular distributions on rational normal scrolls and weighted projective spaces. 
Additionally, under specific conditions, we prove that the singular set of a codimension one holomorphic foliation on a compact toric orbifold admits at least one irreducible component of codimension two, and we also present a Darboux-Jouanolou type integrability theorem
for codimension one holomorphic foliations. Our results are exemplified through various illustrative examples. 
\end{abstract}

\maketitle


\section{Introduction}

In this paper, we are interested in the study of codimension one holomorphic distributions on compact toric orbifolds.
For this we will consider homogeneous coordinates \cite{Cox}.
Toric varieties form an important class of examples in algebraic geometry. Furthermore, its geometry
is fully determined by the combinatorics of its associated fan, which often makes computations
far more tractable. 

The study of holomorphic foliations on toric varieties is a recent topic and there are so many authors studying it, 
see for instance \cite{Mau,Ar1,Ar2,Ro,Rub,Ar3}. 
We remark that holomorphic distributions of codimension one have been given by several authors on certain varieties, such as: 
complex projective spaces \cite{OmCoJa,AGA,GJM,Jou}, Fano manifolds \cite{Ca,CCM}, and general complex manifolds \cite{Dem,Iza}.

Among the main applications in this paper, we study the codimension one regular holomorphic distributions on a compact toric orbifold.
It is worth mentioning the following Thurston characterization for $C^{\infty}$ regular codimension one foliations (in the real context):

\begin{teo} \cite{Tu}
Let $M$ be a closed connected smooth manifold. Then $M$ admits a $C^{\infty}$ regular codimension one foliation if
and only if $\chi(M)=0$.
\end{teo}

Regular distributions are quite rare, for example, on the complex projective space we have: 

\begin{teo} \cite[p.36]{Coo}, \cite{GHS} 
A holomorphic distribution $\mathcal{D}$ on $\mathbb{P}^{n}$ is non singular if and only if $n$ is odd, 
$\mathrm{codim}(\mathcal{D})=1$ and $\deg(\mathcal{D})=0$.
\end{teo}


For surfaces, M. Brunella completely classified regular foliations on complex projective surfaces $S$ with Kodaira 
dimension $\kappa(S)<2$. For surfaces of general type, the classification is open.
As a corollary of this classification we have:

\begin{cor} \cite{Bru}
Let $\mathcal{F}$ be a regular one dimensional holomorphic foliation on a smooth projective rational surface $S$. 
Then $S$ is an Hirzebruch surface and $\mathcal{F}$ is induced by a $\mathbb{P}^{1}$-fibration $S\rightarrow\mathbb{P}^{1}$.
\end{cor}

In higher dimensions, there is no classification of regular holomorphic foliations. 
However, Touzet proposed the following generalization of corollary above:

\begin{conj} \cite{Dru}
Let $\mathcal{F}$  be a regular holomorphic foliation on a rationally connected projective manifold $X$. 
Then $\mathcal{F}$ is induced by a fibration $X\rightarrow Y$ onto a projective manifold.
\end{conj}

\noindent Remember that $X$ is rationally connected if any two points of $X$ can be connected by a rational curve 
(for example, Fano manifolds). The conjecture was verified for weak Fano manifolds in \cite{Dru}, and is 
open already in dimension $3$. However, for a regular codimension one foliation on
a projective threefold, the Minimal Model Program can be used to greatly reduce the problem; the conjecture is proved 
in some special cases, and reduces the problem to understanding regular foliations with nef canonical class. 
For more details see \cite{AF,Dru,Fig}. 

It is important to note that the class of regular and nonintegrable codimension one distributions includes contact structures. 
Contact structures are currently attracting much interest, both from mathematicians and physicists, 
and their classification is still in progress. In \cite{Dru2} Druel proved that toric contact manifolds are either $\mathbb{P}^{2n+1}$ 
or $\mathbb{P}(\mathcal{T}_{\mathbb{P}^1\times\cdots\times\mathbb{P}^1})$. 
Several authors have studied complex contact structures, we refer \cite{Ca,Dru2,KPS} and references therein. \\

The paper is organized as follows: in Section $2$, we start with some preliminaries on toric varieties and orbifolds.
In Section $3$, after defining codimension one holomorphic distributions and foliations, we present two propositions in analogy
with results well-known on the complex projective spaces: in Proposition \ref{proa}, we give a condition for the singular set of a codimension one holomorphic foliation admits at least a codimension two irreducible component (in the case of the complex projective space, this result is due to Jouanolou \cite{Jou}), and in Proposition \ref{prob}, we give a Darboux-Jouanolou type integrability theorem for codimension one holomorphic foliations. Later, we present two Bott-type formulas for the number of singularities, counted with multiplicity, of a generic distribution, i.e. when the distribution has at most isolated singularities; see for instance Proposition 
\ref{teo1} and Proposition \ref{teo20}. 
To end this section, we present some examples where our results are applied. Finally, in Section $4$, we give some applications; 
we impose conditions for a distribution to be regular, and then we present some examples and applications, among them, we give a classification of regular distributions on rational normal scrolls (see Theorem \ref{A} and Corollary \ref{B}), 
and on weighted projective spaces (see Theorem \ref{wei}). 
We observe that, in particular, Theorem \ref{A} verifies Touzet's conjecture for rational normal scrolls, see also Corollary \ref{Me}.
Moreover, in Example $\ref{expr}$, we present a family of contact structures on weighted projective spaces.


\section{Preliminaries}

In this section, we recall some basic definitions and results about complete simplicial toric
varieties. For more details about toric varieties see \cite{Bras}, \cite{Cox2}, \cite{Cox}, \cite{Fu}, and \cite{Oda}.

Let $N\simeq\mathbb{Z}^n$ be a free $\mathbb{Z}$-module of rank $n$ and $M=\mathrm{Hom}_{\mathbb{Z}}(N, \mathbb{Z})$ be its dual. 
A \emph{toric variety} associated with a fan 
$\Delta=\{\sigma_1,\dots,\sigma_k\}$ in $N_{\mathbb{R}}=N\otimes_{\mathbb{Z}} \mathbb{R}\simeq \mathbb{R}^n$, 
is a topological space 
$$\mathbb{P}_{\Delta} := \big(\bigsqcup \, \mathcal{U}_{\sigma_i}\big) / \sim$$ 
endowed with an open covering by the affine toric varieties
$\mathcal{U}_{\sigma_i} = \mathrm{Spec} \, \mathbb{C}[\check{\sigma_i}\cap M]$ for each cone $\sigma_i\in\Delta$, where
$\sigma_i\prec\sigma_j$ ($\sigma_i$ is a face of $\sigma_j$) implies that $\mathcal{U}_{\sigma_i}$ is embedded into $\mathcal{U}_{\sigma_j}$ 
as an open subset; then, the varieties $\mathcal{U}_{\sigma_i}$ are glued together according to this rule into the 
toric variety $\mathbb{P}_{\Delta}$. It is a normal algebraic variety whose charts are defined by binomial relations.

A toric variety $\mathbb{P}_{\Delta}$ determined by a complete simplicial fan $\Delta$ is a \emph{compact complex orbifold}, i.e. a compact complex variety with at most quotient singularities; see Section $\ref{orbif}$. 
It is possible to show that an $n$-dimensional toric variety $\mathbb{P}_{\Delta}$ contain a complex torus $\mathbb{T}^n=(\mathbb{C}^*)^n$ 
as a Zariski open subset such that the action of $\mathbb{T}^n$ on itself extends to an action of $\mathbb{T}^n$ on $\mathbb{P}_{\Delta}$.

\subsection{The homogeneous coordinate ring \cite{Cox}} 

Let $\mathbb{P}_{\Delta}$ be the toric variety determined by a fan
$\Delta$ in $N_{\mathbb{R}}$. The one-dimensional cones of $\Delta$ form the set
$\Delta(1)$. We will assume that $\Delta(1)$ spans $N_{\mathbb{R}}$.

Each $\rho \in \Delta(1)$ corresponds to an irreducible 
$\mathbb{T}$-invariant Weil divisor $D_{\rho}$ in
$\mathbb{P}_{\Delta}$, where
$\mathbb{T}=N\otimes_{\mathbb{Z}}\mathbb{C}^{*} \simeq Hom_{\mathbb{Z}}(M,\mathbb{C}^*)$ is the torus acting
on $\mathbb{P}_{\Delta}$. The $\mathbb{T}$-invariant Weil divisors
on $\mathbb{P}_{\Delta}$ form a free abelian group of rank $|{\Delta}(1)|$, that will be denoted $\mathbb{Z}^{{\Delta}(1)}$. 
Thus an element $D\in \mathbb{Z}^{{\Delta}(1)}$ is a formal sum $D=\sum_{\rho}a_{\rho}D_{\rho}$.
The Weil divisor class group of $\mathbb{P}_{\Delta}$ is denoted by $\mathcal{A}_{n-1}(\mathbb{P}_{\Delta})$.
 
For each $\rho\in \Delta(1)$, introduce a variable $z_{\rho}$,
and consider the polynomial ring
$$\mathrm{S}=\mathbb{C}[z_{\rho}]=\mathbb{C}[z_{\rho} : \rho\in\Delta(1)].$$
Note that a monomial $\prod_{\rho}z_{\rho}^{a_{\rho}}$
determines a divisor $D=\sum_{\rho}a_{\rho}D_{\rho}$, and we set $z^{D}=\prod_{\rho}z_{\rho}^{a_{\rho}}$.
The degree of a monomial $z^{D}\in \mathrm{S}$ is defined by
$\deg(z^{D}) = [D]\in \mathcal{A}_{n-1}(\mathbb{P}_{\Delta})$. 
It is shown that 
$$\mathrm{S}=\bigoplus_{\alpha \in A_{n-1}(\mathbb{P}_{\Delta})}\mathrm{S}_{\alpha},$$
where
$\mathrm{S}_{\alpha}=\bigoplus_{\deg(z^{D})=\alpha}\mathbb{C} \cdot z^{D}$, with $\mathrm{S}_{\alpha}\cdot \mathrm{S}_{\beta}
=\mathrm{S}_{\alpha+\beta}$. 
The polynomial ring $\mathrm{S}$ is called \emph{homogeneous coordinate ring} of the toric variety $\mathbb{P}_{\Delta}$.

Denote by $\mathcal{O}_{\mathbb{P}_{\Delta}}$ the structure sheaf of $\mathbb{P}_{\Delta}$. Let $\mathcal{O}_{\mathbb{P}_{\Delta}}(D)$ be the coherent sheaf on $\mathbb{P}_{\Delta}$ determined by a Weil divisor $D$. If $\alpha = [D] \in \mathcal{A}_{n-1}(\mathbb{P}_{\Delta})$, then
$$\mathrm{S}_{\alpha}\simeq \mathrm{H}^0(\mathbb{P}_{\Delta},\mathcal{O}_{\mathbb{P}_{\Delta}}(D)).$$
If $\mathbb{P}_{\Delta}$ is a complete toric variety, then $\mathrm{S}_{\alpha}$ is finite dimensional for every $\alpha$, 
and in particular, $\mathrm{S}_0 = \mathbb{C}$.

\subsection{The toric homogeneous coordinates}

Given a toric variety $\mathbb{P}_{\Delta}$ of dimension $n$, 
the Weil divisor class group $\mathcal{A}_{n-1}(\mathbb{P}_{\Delta})$ is a finitely generated
abelian group of rank $r=k-n$, where $k=|\Delta(1)|$. 
For each cone $\sigma \in \Delta$, we get the monomial
$$z^{\widehat{\sigma}}=\prod_{\rho\notin \sigma(1)}z_{\rho},$$
where $\sigma(1)=\{\rho\in \Delta(1)\,:\,\rho \prec \sigma\}$.
Then define $\mathcal{Z}=V(\{z^{\widehat{\sigma}}\,:\,\sigma \in \Delta\})\subset \mathbb{C}^{\Delta(1)}$.
We have that $\mathcal{Z} \subset \mathbb{C}^{\Delta(1)}$ has codimension at least two, and 
$\mathcal{Z}=\left\{0\right\}$ when $r=1$; see \cite{Cox1}. Moreover, the group 
$G = Hom_{\mathbb{Z}}(\mathcal{A}_{n-1}(\mathbb{P}_{\Delta}), \mathbb{C}^*)$ 
acts naturally on $\mathbb{C}^{\Delta(1)}$ and leaves $\mathcal{Z}$ invariant.

\begin{teo}\cite{Cox}\label{Cox}
Let $\mathbb{P}_{\Delta}$ be an $n$-dimensional toric variety such that $\Delta(1)$
spans $N_{\mathbb{R}}$. Then $\mathbb{P}_{\Delta}$ is the geometric quotient $(\mathbb{C}^{\Delta(1)}-{\mathcal{Z}})/G$ if and only if
$\mathbb{P}_{\Delta}$ is an orbifold.
\end{teo}

The field of rational functions $K(\mathbb{P}_{\Delta})$ on $\mathbb{P}_{\Delta}$ is the
subfield of $\mathbb{C}(z_{\rho})=\mathrm{Frac}(\mathbb{C}[z_{\rho}])$ given by
$$K(\mathbb{P}_{\Delta})=\left\{\frac{P}{Q}\in
\mathbb{C}(z_{\rho}):P\in \mathrm{S}_{\alpha}, Q\in \mathrm{S}_{\alpha}\right\}.$$
It follows that the polynomials $P,Q \in \mathrm{S}_{\alpha}$ define a rational function 
$\frac{P}{Q}:\mathbb{P}_{\Delta}\dashrightarrow \mathbb{P}^1$.

Assume that $\mathbb{P}_{\Delta}$ is a complex orbifold. We say that $f\in \mathrm{S}_{\alpha }$ is \emph{quasi-homogeneous} of
degree $\alpha$. It follows that the equation $V=\left\{f=0\right\}$ is a
well-defined hypersurface in $\mathbb{P}_{\Delta}$. We say that $V$ is a quasi-homogeneous hypersurface
of degree $\alpha$. 

Given $\rho\in\Delta(1)$, denote by $n_\rho$ the unique generator of $\rho\,\cap N$. 
Consider the $r=k-n$ linearly independent over $\mathbb{Z}$ relations between the $n_{\rho}$ 
of type $\sum_{\rho}a_{i \rho} n_{\rho}=0$, $i=1,\ldots,r$. 
Then, there are $r$ vector fields $R_i=\sum_{\rho}a_{i \rho}z_{\rho}\frac{\partial}{\partial z_{\rho}}$, $i=1,\ldots,r$, 
tangent to the orbits of $G$ and $\mathrm{Lie}(G)=\langle R_1,\dots,R_r\rangle$. 
We will call these vector fields $R_i$, the radial vector fields on $\mathbb{P}_{\Delta}$; for more details, see \cite{Cox1}. \\

Let $\mathbb{P}_{\Delta}$ be a complete simplicial toric variety  of dimension $n$ where $\Delta(1)$
spans $N_{\mathbb{R}}$. We know that $|\Delta(1)|=n+r$, where $r$ is the 
rank of the finitely generated abelian group $\mathcal{A}_{n-1}(\mathbb{P}_{\Delta})$. 
We will denote $\Delta(1)=\left\{\rho_1,\dots,\rho_{n+r}\right\}$, $D_i=D_{\rho_i}$ and $z_i=z_{\rho_i}$
for all $i=1,\dots,n+r$. Then $\mathrm{S}=\mathbb{C}\left[z_1,\dots,z_{n+r}\right]$.

\subsection{Examples}

\begin{enumerate}
	\item \label{exe1} \textbf{Weighted projective spaces.} \cite{Cox2} Let $\omega_0,\dots,\omega_n$ be positive integers with 
$gcd(\omega_0,\dots,\omega_n)=1$. Set $\omega=(\omega_{0},\dots,\omega_{n})$. 
Consider the $\mathbb{C}^*$ action on $\mathbb{C}^{n+1}$ given by 
$$t\cdot(z_0,\dots,z_n)=(t^{\omega_0} z_0,\dots, t^{\omega_n} z_n).$$
Then, the weighted projective spaces is defined by
$$\mathbb{P}(\omega) = \left(\mathbb{C}^{n+1} - \{0\}\right) / \mathbb{C}^*.$$ 
If $\omega_0=\cdots=\omega_n=1$, then $\mathbb{P}(\omega)=\mathbb{P}^{n}$. When $\omega_0,\dots,\omega_n$ 
are pairwise coprime, we say that $\mathbb{P}(\omega)$ is a well formed weighted projective space and we have 
$$\mathrm{Sing}(\mathbb{P}(\omega))=\left\{\bar{e}_i\,|\,\omega_i>1\right\}.$$

Moreover, we have $\mathcal{A}_{n-1}(\mathbb{P}(\omega))\simeq\mathbb{Z}$ and $\deg(z_i)=\omega_i$ for all $0\leq i\leq n+1$.
Consequently the homogeneous coordinate ring of $\mathbb{P}(\omega)$ is given by
$\mathrm{S}=\oplus_{\alpha\geq 0}\mathrm{S}_{\alpha}$, where 
$$\mathrm{S}_{\alpha}=\bigoplus_{p_0\omega_0+\cdots+p_n\omega_n=\alpha} \mathbb{C}\cdot {z_0}^{p_0}\dots\,{z_n}^{p_n}.$$

\vskip 0.2cm

 \item \label{exe2} \textbf{Multiprojective spaces.} \cite{Cox2}
Consider $\mathbb{C}^{n+1}\times\mathbb{C}^{m+1}$ in coordinates $(z_1,z_2)=(z_{1,0},\ldots,z_{1,n},$ $z_{2,0},\ldots,z_{2,m})$. 
Now, consider the $(\mathbb{C}^*)^2$ action on $\mathbb{C}^{n+1}\times\mathbb{C}^{m+1}$ given by 
$$(\mu,\lambda)\cdot(z_1,z_2)=(\mu z_{1,0},\dots,\mu z_{1,n},\lambda z_{2,0},\dots,\lambda z_{2,m}).$$
Then, the multiprojective spaces is defined by 
$$\mathbb{P}^n\times\mathbb{P}^m=\left(\mathbb{C}^{n+1}\times\mathbb{C}^{m+1}-\mathcal{Z}\right)\,/\,(\mathbb{C}^{*})^2,$$
where $\mathcal{Z}=\left(\left\{0\right\}\times\mathbb{C}^{m+1}\right) \cup \left(\mathbb{C}^{n+1}\times\left\{0\right\}\right)$.
Moreover, we have $\mathcal{A}_{n+m-1}(\mathbb{P}^n\times\mathbb{P}^m)\simeq\mathbb{Z}^2$,
$\deg(z_{1,i})=(1,0)$ for all $0\leq i\leq n$, and $\deg(z_{2,j})=(0,1)$ for all $0\leq j\leq m$.
Consequently the homogeneous coordinate ring of $\mathbb{P}^n\times\mathbb{P}^m$ is given by 
$\mathrm{S}=\oplus_{\alpha,\,\beta\geq 0}\mathrm{S}_{(\alpha,\beta)}$, where 
$$\mathrm{S}_{(\alpha,\beta)}=\bigoplus_{p_0+\cdots+p_n=\alpha\,;\,q_0+\cdots+q_m=\beta} \mathbb{C}\cdot {z_{1,0}}^{p_0}\dots\,{z_{1,n}}^{p_n} 
{z_{2,0}}^{q_0}\dots\,{z_{2,m}}^{q_m},$$
is the ring of bihomogeneous polynomials of bidegree $(\alpha,\beta)$. More generally, we can define 
$\mathbb{P}^{n_1}\times\cdots\times\mathbb{P}^{n_k}$ (in general, a finite product of toric varieties is again a toric variety). \\

 \item \label{exe3} \textbf{Hirzebruch surfaces.} \cite{Cox2} Let $r\geq 0$ be an integer. The Hirzebruch surface $\mathcal{H}_r$
is a rational ruled surface defined in $\mathbb{P}^1\times\mathbb{P}^2$ by
$$\big\{\big((\lambda_0:\lambda_1),(\mu_0:\mu_1:\mu_2)\big)\,\,|\,\,\lambda_0^r\mu_0=\lambda_1^r\mu_1\big\}.$$
In particular $\mathcal{H}_0\simeq\mathbb{P}^1\times\mathbb{P}^1$ and 
$\mathcal{H}_1\simeq\mathrm{Bl}_p(\mathbb{P}^2)$ (Blow-up of $\mathbb{P}^2$ at a point $p$).

Consider the $(\mathbb{C}^*)^2$ action on $\mathbb{C}^2\times\mathbb{C}^2$ given as follows 
$$(\lambda,\mu)(z_{1,1},z_{1,2},z_{2,1},z_{2,2})=(\lambda z_{1,1}, \mu z_{1,2}, \lambda z_{2,1},\mu\lambda^r z_{2,2}).$$
The quotient representation of the Hirzeburch surface is given by 
$$\mathcal{H}_r=(\mathbb{C}^2\times\mathbb{C}^2-\mathcal{Z})\,/\,(\mathbb{C}^*)^2,$$
where $\mathcal{Z}=Z(z_{1,1},z_{2,1}) \cup Z(z_{1,2},z_{2,2})$.
Moreover, we have $\mathcal{A}_{1}(\mathcal{H}_r)\simeq\mathbb{Z}^2$ and the homogeneous coordinate ring associated to 
$\mathcal{H}_r$ is given by $\mathrm{S}=\oplus_{\alpha\geq 0,\,\beta\geq 0}\mathrm{S}_{(\alpha,\beta)}$, where 
$$\mathrm{S}_{(\alpha,\beta)}=\bigoplus_{\alpha=p_1+q_1+rq_2\,;
\,\beta=p_2+q_2} \mathbb{C}\cdot {z_{1,1}}^{p_1}{z_{1,2}}^{p_2}{z_{2,1}}^{q_1}{z_{2,2}}^{q_2}.$$
In particular $\deg(z_{1,1})=(1,0)$, $\deg(z_{1,2})=(0,1)$, $\deg(z_{2,1})=(1,0)$ and $\deg(z_{2,2})=(r,1)$. \\

 \item \label{exe4} \textbf{Rational normal scrolls.} 
Let $a_1,\ldots,a_n$ be integers. Consider the $(\mathbb{C}^*)^2$ action on $\mathbb{C}^2\times\mathbb{C}^n$ given by 
$$(\lambda,\mu)(z_{1,1},z_{1,2},z_{2,1},\dots,z_{2,n})=(\lambda z_{1,1}, \lambda z_{1,2}, \mu\lambda^{-a_1} z_{2,1},
\dots,\mu\lambda^{-a_n} z_{2,n}).$$
Then, an $n$-dimensional rational normal scroll $\mathbb{F}(a_1,\dots,a_n)$ is a projective toric manifold defined by 
$$\mathbb{F}(a_1,\dots,a_n)=(\mathbb{C}^2\times\mathbb{C}^n-\mathcal{Z})\,/\,(\mathbb{C}^*)^2,$$
where $\mathcal{Z}=\left(\left\{0\right\}\times\mathbb{C}^{n}\right) \cup \left(\mathbb{C}^{2}\times\left\{0\right\}\right)$.
Moreover, we have $\mathcal{A}_{n-1}(\mathbb{F}(a_1,\dots,a_n))\simeq\mathbb{Z}^2$ and the homogeneous coordinate ring associated to 
$\mathbb{F}(a_1,\dots,a_n)$ is given by $\mathrm{S}=\oplus_{\alpha\in\mathbb{Z},\,\beta\geq 0}\mathrm{S}_{(\alpha,\beta)}$, where 
$$\mathrm{S}_{(\alpha,\beta)}=\bigoplus_{\alpha=p_1+p_2 - \sum_iq_ia_i\,;
\,\beta=\sum_iq_i} \mathbb{C}\cdot {z_{1,1}}^{p_1}{z_{1,2}}^{p_2}{z_{2,1}}^{q_1}\dots\,{z_{2,n}}^{q_n}.$$
In particular $\deg(z_{1,1})=\deg(z_{1,2})=(1,0)$ and $\deg(z_{2,i})=(-a_i,1)$.

Consider  $1\leq a_1 \leq a_2 \leq \dots \leq a_n$ integers, it is possible to show that $\mathbb{F}(a_1,\dots,a_n)\simeq\mathbb{P}(\mathcal{O}_{\mathbb{P}^1}(a_1)\oplus\dots\oplus\mathcal{O}_{\mathbb{P}^1}(a_n))$. For more details see \cite{Cox2, Re}.
\end{enumerate}

\subsection{Orbifolds} \label{orbif}\cite{Rua, Bla, Sata, Ma} An $n$-dimensional \emph{orbifold} $X$ is a complex space endowed with the following property: each point $x\in X$ possesses a open neighborhood, which is the quotient $U_x=\widetilde{U}/G_x$, where $\widetilde{U}$
is a complex manifold of dimension $n$, and $G_x$ is a properly discontinuous finite group of automorphisms of $\widetilde{U}$.
Thus, locally we have a quotient map $\pi_x:(\widetilde{U},\tilde{x})\rightarrow(U_x,x)$.

Since the possible singularities appearing in an orbifold $X$ are quotient singularities, we have that $X$ is reduced, normal, Cohen-Macaulay,  and with only rational singularities. 

Satake's fundamental idea was to use local smoothing coverings to extend the
definition of usual known objects to orbifolds. For instance, smooth differential $k$-forms
on $X$ are $C^{\infty}$-differential $k$-forms $\omega$ on $X_{\mathrm{reg}}=X \setminus \mathrm{Sing}(X)$ such that the pull-back
$\pi^{\ast}_x\omega$ extends to a $C^{\infty}$-differential $k$-form on every local smoothing covering 
$\pi_x:(\widetilde{U},\tilde{x})\rightarrow(X,x)$ of $X$. Hence, if $\omega$ is a smooth $2n$-form on $X$ with compact support 
$\mathrm{Supp}(\omega)\subset(X, x)$, then by definition
$$\int_X^{orb} \omega=\frac{1}{\# G_x}\int_{\widetilde{U}}\pi^{\ast}\omega.$$

If now $\omega$ has compact support, we use a partition of unity 
$\left\{\rho_{\alpha},U_{\alpha}\right\}_{\alpha\in\Lambda}$, where
$\pi_{\alpha}:(\widetilde{U}_{\alpha},\tilde{x}_{\alpha})\rightarrow(U_{\alpha},x_{\alpha})$ is a local smoothing covering and
$\sum \rho_{\alpha}(x)=1$ for all $x\in\mathrm{Supp}(\omega)$, and set
$$\int_X^{orb} \omega=\sum_{\alpha}\int_{X}^{orb}\rho_{\alpha}\omega.$$

\begin{obs}\cite{Ma} Let $X$ be compact orbifold and 
$Ker(X)=\left\{g\in\coprod_{\alpha\in\Lambda}G_{\alpha}:\,g\cdot x=x,\,\forall x\in X\right\}$. Then 
$$\int_X^{orb} \omega=\frac{1}{\# Ker(X)}\int_{X_{reg}}\omega.$$
\end{obs}

Analogous to the manifold case, we have the sheaf of $C^{\infty}$-differential $k$-forms and
exterior differentiation, moreover, the concepts of connection and curvature are well defined. 
Stokes' formula, Poincar\'e's duality and Rham's theorem also hold on orbifolds.


\section{Codimension one holomorphic distributions}

There exists an exact sequence known as the generalized Euler's sequence 
$$0\rightarrow \mathcal{O}_{\mathbb{P}_{\Delta}}^{\oplus
r}\rightarrow\bigoplus_{i=1}^{n+r}\mathcal{O}_{\mathbb{P}_{\Delta}}(D_i)\rightarrow
\mathcal{T}\mathbb{P}_{\Delta}\rightarrow0,$$
\noindent where
$\mathcal{T}\mathbb{P}_{\Delta}=\mathcal{H}om(\Omega_{\mathbb{P}_{\Delta}}^1,\mathcal{O}_{\mathbb{P}_{\Delta}})$ is the so-called Zariski tangent sheaf of $\mathbb{P}_{\Delta}$. 
Let $i:\mathbb{P}_{\Delta \,\mathrm{reg}} \rightarrow \mathbb{P}_{\Delta}$
be the inclusion of the regular part $\mathbb{P}_{\Delta \,\mathrm{reg}}= \mathbb{P}_{\Delta} - \mathrm{Sing}(\mathbb{P}_{\Delta})$.
Since $\mathbb{P}_{\Delta}$ is a complex orbifold then $\mathcal{T}\mathbb{P}_{\Delta} \simeq 
i_{\ast}\mathcal{T}\mathbb{P}_{\Delta \,\mathrm{reg}}$, where $\mathcal{T}\mathbb{P}_{\Delta \,\mathrm{reg}}$ is the tangent sheaf of $\mathbb{P}_{\Delta \,\mathrm{reg}}$; 
see \cite[Appendix A.2]{Cox3}.\\

Let $\mathcal{O}_{\mathbb{P}_{\Delta}}(d)=\mathcal{O}_{\mathbb{P}_{\Delta}}(d_1,\ldots,d_{n+r})
=\mathcal{O}_{\mathbb{P}_{\Delta}}(\sum_{i=1}^{n+r}d_{i}D_{i})$, 
where $\sum_{i=1}^{n+r}d_{i}D_{i}$ is a Cartier divisor. \\ 
Set
$\Omega^1_{\mathbb{P}_{\Delta}}(d)=\Omega^1_{\mathbb{P}_{\Delta}}\otimes \mathcal{O}_{\mathbb{P}_{\Delta}}(d)$.
Tensorizing the dual Euler's sequence by $\mathcal{O}_{\mathbb{P}_{\Delta}}(d)$ we get
\begin{eqnarray*}
0 \rightarrow \Omega^1_{\mathbb{P}_{\Delta}}(d)
\rightarrow \bigoplus_{i=1}^{n+r}\mathcal{O}_{\mathbb{P}_{\Delta}}(d_1,\dots,d_i-1,\dots,d_{n+r})
\rightarrow Cl(\mathbb{P}_{\Delta})\otimes_{\mathbb{Z}}\mathcal{O}_{\mathbb{P}_{\Delta}}(d) \rightarrow 0.
\end{eqnarray*}
For more details, see for instance \cite{Cox1, Cox2}.

\begin{defi} 
A codimension one holomorphic distribution $\mathcal{D}$ on $\mathbb{P}_{\Delta}$ of degree 
$d=\sum_{i=1}^{n+r}d_{i}\left[D_{i}\right]\in\mathcal{A}_{n-1}(\mathbb{P}_{\Delta})$
is a global section $\omega$ of $\Omega^1_{\mathbb{P}_{\Delta}}(d)$.  
For simplicity of notation we say that $\mathcal{D}$ has degree $d=(d_1,\dots,d_{n+r})$.
A codimension one holomorphic foliation $\mathcal{F}$ on $\mathbb{P}_{\Delta}$ is a distribution that satisfies 
the Frobenius integrability condition $\omega\wedge d\omega=0$.
We will consider codimension one holomorphic foliations whose singular scheme has codimension greater than $1$.
\end{defi}

If
$H^1\big(\mathbb{P}_{\Delta}, \Omega^1_{\mathbb{P}_{\Delta}}(d)\big)=0$; see for instance the Bott-Steenbrink-Danilov vanishing theorem
\cite[Theorem $9.3.1$]{Cox2}, 
then we have that a codimension one holomorphic distribution $\mathcal{D}$ on
$\mathbb{P}_{\Delta}$ of degree $d$ is given by a polynomial form in homogeneous coordinates of the form
$$
\omega=\sum_{i=1}^{n+r}P_i \, dz_i,
$$
where $P_i$ is a quasi-homogeneous polynomial of degree $(d_1,\dots,d_i-1,\dots,d_{n+r})$ for all $i=1,\dots,n+r$, such that
$i_R \, \omega=0$ for all $R\in\mathrm{Lie}(G)$. 
Note that $\omega$ is a quasi-homogeneous form of degree $d$. We define the singular set of $\mathcal{D}$ by 
$$\mathrm{Sing}(\mathcal{D})=\pi\left(\left\{p\in\mathbb{C}^{n+r}:\,\omega(p)=0\right\}\right),$$
where $\pi:\mathbb{C}^{n+r}-{\mathcal{Z}} \rightarrow \mathbb{P}_{\Delta}$ is the canonical projection. 

Let $V = \{f = 0\}$ be a quasi-homogeneous hypersurface. We recall that $V$ is \emph{invariant} by $\mathcal{D}$ if and only if 
there is a holomorphic $2$-form $\Theta$ in homogeneous coordinates, such that $\omega\wedge d f = f\cdot\Theta$. 
A \emph{rational first integral} of $\mathcal{D}$ is a non-constant rational function $f$ such that $\omega\wedge df=0$.

\subsection{Examples}

\begin{enumerate}
	\item \textbf{Weighted projective spaces.}
The Euler's sequence on $\mathbb{P}(\omega)$ is an exact sequence of orbibundles
$$
0\longrightarrow
\underline{\mathbb{C}}\longrightarrow\bigoplus_{i=0}^{n
}\mathcal{O}_{\mathbb{P}(\omega)}(\omega_i) \longrightarrow
T\mathbb{P}(\omega) \longrightarrow 0,
$$
where $\underline{\mathbb{C}}$ is the trivial line orbibundle on
$\mathbb{P}(\omega)$. The radial vector field is given by
$$R = \omega_0 z_0 \frac{\partial}{\partial z_0}+\cdots+ \omega_n z_n \frac{\partial}{\partial z_n}.$$
 \item \textbf{Multiprojective spaces.} The Euler's sequence on $\mathbb{P}^{n}\times\mathbb{P}^{m}$ is
$$0\longrightarrow \underline{\mathbb{C}^{2}}\,
{\longrightarrow}
\mathcal{O}_{\mathbb{P}^{n}\times\mathbb{P}^{m}}(1,0)^{\oplus n+1}\oplus\mathcal{O}_{\mathbb{P}^{n}\times\mathbb{P}^{m}}(0,1)^{\oplus m+1}
{\longrightarrow}\,
T(\mathbb{P}^{n}\times\mathbb{P}^{m})
\longrightarrow 0.$$
Here, the radial vector fields are given by
$$R_1 = z_{1,0} \frac{\partial}{\partial z_{1,0}}+\cdots+ z_{1,n} \frac{\partial}{\partial z_{1,n}},\,\,\mbox{and}\,\,\,
R_2 = z_{2,0} \frac{\partial}{\partial z_{2,0}}+\cdots+ z_{2,m} \frac{\partial}{\partial z_{2,m}}.$$

 \item \textbf{Hirzebruch surface.} The Euler's sequence on $\mathcal{H}_r$ is
$$0 \rightarrow\mathcal{O}_{\mathcal{H}_r}^{\oplus2}\rightarrow\mathcal{O}_{\mathcal{H}_r}(1,0)\oplus\mathcal{O}_{\mathcal{H}_r}(0,1)\oplus\mathcal{O}_{\mathcal{H}_r}(1,0)\oplus\mathcal{O}_{\mathcal{H}_r}(r,1)\rightarrow \mathcal{T}\mathcal{H}_r\rightarrow0.$$
Here, the radial vector fields are given by
$$R_1 = z_{1,1}\frac{\partial}{\partial z_{1,1}}+ z_{2,1}\frac{\partial}{\partial z_{2,1}} 
+ r z_{2,2}\frac{\partial}{\partial z_{2,2}},\,\,\mbox{and}\,\,\, 
R_2 = z_{1,2}\frac{\partial}{\partial z_{1,2}}+ z_{2,2}\frac{\partial}{\partial z_{2,2}}.$$

 \item \textbf{Rational normal scrolls.} The Euler's sequence on $\mathbb{F}(a):=\mathbb{F}(a_1,\dots,a_n)$ is
$$0 \rightarrow\mathcal{O}_{\mathbb{F}(a)}^{\oplus
2}\rightarrow\mathcal{O}_{\mathbb{F}(a)}(1,0)^{\oplus 2}\oplus\bigoplus_{i=1}^n
\mathcal{O}_{\mathbb{F}(a)}(-a_i,1)\rightarrow \mathcal{T}\mathbb{F}(a)\rightarrow0.$$
Here, the radial vector fields are given by
$$R_1 = z_{1,1}\frac{\partial}{\partial z_{1,1}}+ z_{1,2}\frac{\partial}{\partial z_{1,2}} 
- \sum_{i=1}^{n}a_i z_{2,i}\frac{\partial}{\partial z_{2,i}},\,\,\mbox{and}\,\,\, 
R_2 = \sum_{i=1}^{n}z_{2,i}\frac{\partial}{\partial z_{2,i}}.$$
 \end{enumerate}

\bigskip

The proposition below is attributed to Jouanolou \cite[Proposition $2.6$]{Jou}, specifically in the context of the complex projective space.

\begin{prop} \label{proa} Let $\mathbb{P}_{\Delta}$ be a complete simplicial toric variety of dimension $n$, with homogeneous coordinates 
$z_1,\ldots,z_{n+r}$.
Let $\omega\in\Omega^1_{\mathbb{P}_{\Delta}}(d)$ be nonzero integrable, 
and suppose there exists a radial vector field 
$R=\sum_ia_iz_i\frac{\partial}{\partial z_i}$ on $\mathbb{P}_{\Delta}$ such that $\sum_ia_id_i\neq0$. 
Then, the singular set of $\omega$ must contain at least a codimension two irreducible component. 
\end{prop}
\begin{proof} Let us first prove the following identity
$$i_R \left(d\omega\right)=\big(\Sigma_ia_id_i\big)\omega.$$
Let $R_t(z)=\left(e^{a_1t}z_1,\ldots,e^{a_{n+r}t}z_{n+r}\right)$ be the flow of $R$; $R_0(z)=z$ and $R_0'(z)=R(z)$. Set 
$\omega=\sum_i P_i(z)dz_i$. Then 
\begin{eqnarray*}
\mathcal{L}_R(\omega)&=&\frac{d}{d t}R^{\ast}_t(\omega)\Big|_{t=0}=
\frac{d}{d t}\sum_iP_i\left(e^{a_1t}z_1,\ldots,e^{a_{n+r}t}z_{n+r}\right)d\left(e^{a_it}z_i\right)\Big|_{t=0} \\
&=&\frac{d}{d t}\left(e^{\left(\sum_ia_id_i\right)t}\right)\Big|_{t=0}\cdot\sum_iP_i(z)dz_i=\big(\sum_ia_id_i\big)\omega,
\end{eqnarray*}
and consequently we have $(\sum_ia_id_i)\omega=\mathcal{L}_R(\omega)=i_R(d\omega)+d(i_R\omega)=i_R(d\omega)$.
Now, the proposition follows from Saito's division Lemma \cite{Jou, Sai}. In fact, suppose that $\mathrm{codim\,Sing}(\omega)\geq 3$, the condition $\omega\wedge d\omega=0$ implies that $d\omega=\alpha\wedge\omega$ for some polynomial $1$-form $\alpha$, and so 
$d\omega=0$. Therefore we have $\left(\sum_i a_id_i\right)\omega=i_R(d\omega)=0$, an absurd.
\end{proof}

Let us prove the following result due to Darboux, in the case of the complex projective space.

\begin{prop} \label{prob} Let $\mathbb{P}_{\Delta}$ be a complete simplicial toric variety of dimension $n$, with homogeneous coordinates $z_1,\ldots,z_{n+r}$ and Picard number $r=rank \,\mathcal{A}_{n-1}(\mathbb{P}_{\Delta})$.
Let $\mathcal{F}$ be a codimension one holomorphic foliation of degree $d$ on $\mathbb{P}_{\Delta}$. Set 
$h_{\alpha}=\dim_{\mathbb{C}}H^0(\mathbb{P}_{\Delta}, \mathcal{O}_{\mathbb{P}_{\Delta}}(\alpha))=\dim_{\mathbb{C}} \mathrm{S}(\Delta)_{\alpha}$.
If $\mathcal{F}$ admits 
$$\mathcal{N}_{\Delta}=2+\sum_{1\leq i<j \leq n+r} h_{d-\deg(z_i)-\deg(z_j)}$$
invariant irreducible quasi-homogeneous hypersurfaces, then $\mathcal{F}$ admits a rational first integral.

In particular, $\mathcal{F}$ admits a rational first integral if and only if it admits an infinite number of
invariants irreducible quasi-homogeneous hypersurfaces.
\end{prop}

\begin{proof} The proof is adapted from \cite[Theorem $1.3$]{Ne}.
Set $N=\mathcal{N}_{\Delta}$. Let $f_1,\ldots,f_{N}$ be the $\mathcal{F}$-invariant irreducible quasi-homogeneous hypersurfaces.
Let $\omega$ be a polynomial $1$-form that defines $\mathcal{F}$ in homogeneous coordinates. 
The invariance condition of $f_i$ implies that 
\begin{equation}\label{eqi}
\frac{df_i}{f_i}\wedge\omega=\Theta_i,\,\,\,i=1,\ldots,N,
\end{equation}
where the $\Theta_i$ are quasi-homogeneous $2$-forms of degree $d$ on $\mathbb{P}_{\Delta}$. 
Since the vector space of quasi-homogeneous $2$-forms of degree $d$ on $\mathbb{P}_{\Delta}$ has finite dimension equal to 
$$\sum_{1\leq i<j \leq n+r} h_{d-\deg(z_i)-\deg(z_j)},$$
the set $\left\{\Theta_1,\ldots,\Theta_{N}\right\}$ is linearly dependent. 
Then there exist complex numbers $\alpha_1,\ldots,\alpha_N$ not all equal to zero, such that $\sum_{i=1}^N \alpha_i \cdot\Theta_i=0$.
By $(\ref{eqi})$ it follows that $\eta_1\wedge\omega=0$, where
$$\eta_1=\sum_{i=1}^N\alpha_i\,\frac{df_i}{f_i}.$$
In particular, the $1$-form $\omega_1=f_1\ldots f_N\cdot\eta_1$ is holomorphic and such that $\omega_1\wedge\omega=0$.
The condition $\mathrm{codim\,Sing}(\omega)\geq 3$ implies there exists $p\in \mathbb{C}\left[z_1,\ldots,z_{n+r}\right]$ such that
$\omega_1=p\cdot\omega$, so we can write $\eta_1=h_1\cdot\omega$, where $h_1=\frac{p}{f_1\ldots f_N}$.
By a similar argument, considering the indices $2,\ldots,N+1$, we can built a $1$-form 
$$\eta_2=\sum_{i=2}^{N+1}\beta_i\,\frac{df_i}{f_i},$$
such that the vectors $\left(\alpha_1,\ldots\alpha_N,0\right)$ and $\left(0,\beta_2,\ldots\beta_{N+1}\right)$ are linearly independent,
and there exists a meromorphic function $h_2$ such that $\eta_2=h_2\cdot\omega$. 
The relation $\eta_1=h_1\cdot\omega$ and $\eta_2=h_2\cdot\omega$ implies that $\eta_2=\frac{h_2}{h_1}\cdot\eta_1$.
From this we obtain $d\big(\frac{h_2}{h_1}\big)\wedge\eta_1=d\eta_2=0$, hence that $d\big(\frac{h_2}{h_1}\big)\wedge\omega=0$, 
and finally that $d\big(\frac{h_2}{h_1}\big)=h_3\cdot\omega$, where $h_3$ is a meromorphic function. 

Let us verify now that $\frac{h_2}{h_1}$ is not constant.
Suppose that $\frac{h_2}{h_1}=c\in\mathbb{C}$, then $\eta_2=h_2\cdot\omega=c\cdot\eta_1$, which is absurd. 
From this we conclude that $\frac{h_2}{h_1}$ is a rational first integral of $\mathcal{D}$. 
\end{proof} 

\begin{obs} Consider the multiprojective space $\mathbb{P}_{\Delta}=\mathbb{P}^{n_1}\times\cdots\times\mathbb{P}^{n_r}$. 
Set $\alpha=(m_1,\ldots,m_r)\in\mathcal{A}_{n-1}(\mathbb{P}_{\Delta})$. Then 
$$h_{\alpha}=\prod_{i=1}^r\binom{n_i+m_i}{n_i}.$$
Similarly, consider the $n$-dimensional rational normal scroll $\mathbb{P}_{\Delta}=\mathbb{F}(a_1,\ldots,a_n)$ and 
$\alpha=(m_1,m_2)\in\mathcal{A}_{n-1}(\mathbb{P}_{\Delta})$. Then 
$$h_{\alpha}=\left(\sum_{i=1}^na_i\right)\binom{m_1+n-1}{n}+\left(m_2+1\right)\binom{m_1+n-1}{n-1},$$
see \cite{Re}.
\end{obs}

\begin{exe} Let $\mathbb{P}(\omega)$ be an $n$-dimensional weighted projective space.
Let $\mathrm{P}_{\mathrm{S}}(t)$ be the Poincar\'e series of the graded algebra $\mathrm{S}(\omega)$ of $\mathbb{P}(\omega)$ given by
$$\mathrm{P}_\mathrm{S}(t)=\sum_{k=0}^{\infty}h_k\,t^{\,k}=\prod_{k=0}^n\big(1-t^{\,\omega_k}\big)^{-1},$$
where $h_{k}=\dim_{\mathbb{C}} \mathrm{S}(\omega)_{k}$; see \cite{Do}.
Let $\mathcal{F}$ be a codimension one holomorphic foliation on $\mathbb{P}(\omega)$ of degree $d$. Then, if $\mathcal{F}$ admits
$$\mathcal{N}(\omega)=2+\sum_{1\leq i<j \leq n+1}h_{d-\omega_i-\omega_j}$$
invariant irreducible quasi-homogeneous hypersurfaces, then $\mathcal{F}$ admits a rational first integral.

In particular, if $\omega=(1,\ldots,1)$ and $\mathcal{F}$ is a codimension one holomorphic foliation on $\mathbb{P}^{n}$ of degree $d+2$, 
then we have
$$h_{d+2-\omega_i-\omega_j}=h_d=\frac{1}{d\,!}\mathrm{P}_\mathrm{S}^{(d)}(0)=\binom{n+d}{n},
\,\,\,\,\mbox{and}\,\,\,\,\,\,\mathcal{N}(\omega)=\binom{n+d}{n}\binom{n+1}{2}+2.$$ 
Consequently we recover the Darboux-Jouanolou type integrability theorem for codimension one holomorphic foliations 
in the projective space. Similarly, for a Darboux-Jouanolou-Ghys type integrability theorem for
one-dimensional foliations on toric varieties, see \cite{Mau}.
\end{exe}


\subsection{Codimension one distribution with isolated singularities}

Set $k=(k_1,\ldots,k_n)\in\mathbb{N}^{n}$, 
$\left|k\right|=\sum k_i$ and $\binom{m}{k}=\frac{m!}{k_1!\cdots k_{n}!}$.

\begin{prop} \label{teo1} Let $\mathbb{P}_{\Delta}$ be an $n$-dimensional compact toric manifold.
Set $h_i=\left[D_i\right]\in\mathcal{A}_{n-1}(\mathbb{P}_{\Delta})$.
Let $\mathcal{D}$ be a codimension one distribution of degree 
$d=\sum_{i}d_{i} h_i$
with isolated zeros on $\mathbb{P}_{\Delta}$. Then, the singular scheme of $\mathcal{D}$ consists of
$$
\#\,\mathrm{Sing}\left(\mathcal{D}\right)
=\sum_{j=0}^{n}\left\{(-1)^j\sum_{\left|k\right|=n-j}\binom{n-j}{k}
\int_{\mathbb{P}_{\Delta}} \texttt{C}_j(h)\cdot \prod_i \big(d_i h_i\big)^{k_i}\right\}
$$
points counted with multiplicity, where $\texttt{C}_{j}(h)$ is the $j$th elementary symmetric 
function of the variables $h_1,\ldots, h_{n+r}$. \\
\end{prop}

\begin{proof}
Using the generalized Euler's sequence, we have 
$$c\left(\mathcal{T}\mathbb{P}_{\Delta}\right)
=c\left(\mathcal{O}^{\oplus r}_{\mathbb{P}_{\Delta}}\right)c\left(\mathcal{T}\mathbb{P}_{\Delta}\right)
=c\left(\bigoplus_{i=1}^{n+r} \mathcal{O}_{\mathbb{P}_{\Delta}}(D_i)\right),$$
where $c$ denotes the total Chern classes. 
Note that $h_i=c_{1}(\mathcal{O}_{\mathbb{P}_{\Delta}}(D_i))=\left[D_i\right]$. Then
\begin{eqnarray*}
c_{n}\big(\mathcal{T}\mathbb{P}_{\Delta}\otimes\mathcal{O}_{\mathbb{P}_{\Delta}}(d)\big) &=&
\sum_{j=0}^{n}c_{j}\big(\mathcal{T}\mathbb{P}_{\Delta}\big)c_{1}\big(\mathcal{O}_{\mathbb{P}_{\Delta}}(d)\big)^{n-j} \nonumber\\
&=&\sum_{j=0}^{n}c_{j}\left(\bigoplus_{i=1}^{n+r} \mathcal{O}_{\mathbb{P}_{\Delta}}(D_i)\right)
\left(\sum_{i=1}^{n+r}c_1\big(\mathcal{O}_{\mathbb{P}_{\Delta}}(d_iD_i)\big)\right)^{n-j} \nonumber\\
&=& \sum_{j=0}^{n}c_{j}\left(\bigoplus_{i=1}^{n+r} \mathcal{O}_{\mathbb{P}_{\Delta}}(D_i)\right)
\left(\sum_{i=1}^{n+r} d_ih_i \right)^{n-j}.
\end{eqnarray*}
On the other hand, we have  
\begin{eqnarray*}
c\left(\bigoplus_{i=1}^{n+r} \mathcal{O}_{\mathbb{P}_{\Delta}}(D_i)\right)=
\prod_{i=1}^{n+r}c\big(\mathcal{O}_{\mathbb{P}_{\Delta}}(D_i)\big)=
\prod_{i=1}^{n+r}\Big(1+c_1\big(\mathcal{O}_{\mathbb{P}_{\Delta}}(D_i)\big)\Big)
=\prod_{i=1}^{n+r}(1+h_{i})=\sum_{j=0}^{n+r}\texttt{C}_{j}(h).
\end{eqnarray*}
Then 
\begin{eqnarray*}
c_{n}\big(\mathcal{T}\mathbb{P}_{\Delta}\otimes\mathcal{O}_{\mathbb{P}_{\Delta}}(d)\big)
=\sum_{j=0}^{n}\texttt{C}_j(h)\left(\sum_{i=1}^{n+r} d_ih_i \right)^{n-j}
=\sum_{j=0}^{n}\left\{\sum_{\left|k\right|=n-j}\binom{n-j}{k}\texttt{C}_j(h)
\prod_i \big(d_ih_i\big)^{k_i}\right\},
\end{eqnarray*}
hence
\begin{eqnarray*}
c_{n}\big(\Omega^1_{\mathbb{P}_{\Delta}}(d)\big)
=(-1)^n c_{n}\big(\mathcal{T}\mathbb{P}_{\Delta}\otimes\mathcal{O}_{\mathbb{P}_{\Delta}}(-d)\big)
=\sum_{j=0}^{n}\left\{(-1)^j\sum_{\left|k\right|=n-j}\binom{n-j}{k}\texttt{C}_j(h)
\prod_i \big(d_i h_i\big)^{k_i}\right\},
\end{eqnarray*}
so we get
\begin{eqnarray*}
\#\,\mathrm{Sing}\left(\mathcal{D}\right)
&=& \int_{\mathbb{P}_{\Delta}} c_{n}\big(\Omega^1_{\mathbb{P}_{\Delta}}(d)\big) 
=\sum_{j=0}^{n}\left\{(-1)^j\sum_{\left|k\right|=n-j}\binom{n-j}{k}
\int_{\mathbb{P}_{\Delta}} \texttt{C}_j(h)\cdot \prod_i \big(d_i h_i\big)^{k_i}\right\}.
\end{eqnarray*}
\end{proof}

\begin{exe} \label{exmul}
Consider the multiprojective space $\mathbb{P}_{\Delta}=\mathbb{P}^n\times\mathbb{P}^m$.
Here $h_1=h_{10}=\ldots=h_{1n}$ and $h_2=h_{20}=\ldots=h_{2m}$, with $h_1^{n+1}=0$ and $h_2^{m+1}=0$. 
Let $\mathcal{D}$ be a codimension one distribution of degree $d=(d_1, d_2)=d_1 h_1+d_2h_2$
with isolated zeros on $\mathbb{P}^n\times\mathbb{P}^m$. Then 
\begin{eqnarray*}
\#\,\mathrm{Sing}\left(\mathcal{D}\right)
&=&\sum_{j=0}^{n+m}\left\{(-1)^j\sum_{\left|k\right|=n+m-j}\binom{n+m-j}{k_1, k_2}
\int_{\mathbb{P}_{\Delta}} \texttt{C}_j(h)\cdot \prod_{i=1}^2 \big(d_i h_i\big)^{k_i}\right\} \\
&=&(-1)^{n+m}\sum_{\substack{0\leq k_1\leq n, \\ 0\leq k_2\leq m}}(-1)^{\left|k\right|}
\binom{\left|k\right|}{k_1, k_2} d_1^{\,k_1} d_2^{\,k_2}
\int_{\mathbb{P}_{\Delta}} \texttt{C}_{n+m-\left|k\right|}(h)\cdot h_1^{k_1}h_2^{k_2} \\
&=&(-1)^{n+m}\sum_{\substack{0\leq k_1\leq n, \\ 0\leq k_2\leq m}}(-1)^{k_1+k_2}\binom{k_1+k_2}{k_1, k_2}
\binom{n+1}{n-k_1}\binom{m+1}{m-k_2} d_1^{\,k_1} d_2^{\,k_2}.
\end{eqnarray*} 
Similarly, for generic one-dimensional foliations, a formula for the number of singularities 
in multiprojective spaces is given in \cite{MaCor}.
\end{exe}

Let us consider an example in blow-ups: 
\begin{exe} It is well-known that the blow-ups $X_r$ of $\mathbb{P}^2$ in $0\leq r\leq 3$ points in
general position are toric del Pezzo surfaces; see for instance \cite{Ot}. 

Let us consider $X_3$; it is called the del Pezzo surface of degree six.
The blow-up morphism \\
$\pi:X_3\rightarrow\mathbb{P}^2$ induces an injective group homomorphism 
$\pi^{\ast}:\mathrm{Pic}\left(\mathbb{P}^2\right)\rightarrow\mathrm{Pic}\left(\mathrm{X_3}\right)$. 
We write $H=\pi^{\ast}L$ where $L$ is a line in $\mathbb{P}^2$, and let us denote by $E_1$, $E_2$ and $E_3$ the exceptional curves. 
Then, we can write $\mathrm{Pic}(X_3)=\mathbb{Z}H\oplus\mathbb{Z}(-E_2)\oplus\mathbb{Z}(-E_1)\oplus\mathbb{Z}(-E_3)\simeq\mathbb{Z}^4$, where
$$H^2=1,\,\,E_i\cdot E_j=-\delta_{ij}\,\,\mbox{and}\,\,H\cdot E_i=0.$$ 
Now, let $x, y, z, s, t, u$ be coordinate functions on $\mathbb{C}^6$. As a toric variety we have 
$$X_3=(\mathbb{C}^{\,6}-\mathcal{Z})\,/\,(\mathbb{C}^*)^4,$$
where $\mathcal{Z}=Z(x,t) \cup Z(y,s) \cup Z(z,u) \cup Z(x,y) \cup Z(y,z) \cup Z(z,x) \cup Z(s,t) \cup Z(u,t) \cup Z(s,u)$.\\
Let us denote by $L_1$, $L_2$ and $L_3$ the strict transforms of the lines in $\mathbb{P}^2$ spanned by the
points that are blown-up. Then we have
$$H=L_1+E_2+E_3=L_2+E_1+E_3=L_3+E_1+E_2,$$
and $h_1=\deg(x)=L_1$, $h_2=\deg(y)=L_2$, $h_3=\deg(z)=L_3$, $h_4=\deg(s)=E_2$, $h_5=\deg(t)=E_1$ and $h_6=\deg(u)=E_3$;
for more details see \cite{Pa}.

\noindent Now, let $\mathcal{F}$ be a holomorphic foliation of degree 
$d=(d_0,d_1,d_2,d_3)=d_0H-d_1E_2-d_2E_1-d_3E_3$ on $X_3$. Since $\texttt{C}_0(h)=1$, $\texttt{C}_1(h)=3H-E_1-E_2-E_3$ 
and $\texttt{C}_2(h)=6$, we have 
$$\#\,\mathrm{Sing}\left(\mathcal{F}\right)=\sum_{j=0}^{2}(-1)^j d^{\,2-j}\texttt{C}_j(h)
=d_0(d_0-3)+d_1(1-d_1)+d_2(1-d_2)+d_3(1-d_3)+6.$$

Similarly, for the blow-up morphism $\pi:X_2\rightarrow\mathbb{P}^2$ (del Pezzo surface of degree seven), 
we can write $\mathrm{Pic}(X_2)=\mathbb{Z}H\oplus\mathbb{Z}(-E_1)\oplus\mathbb{Z}(-E_2)\simeq\mathbb{Z}^3$, where
$H^2=1$, $E_i\cdot E_j=-\delta_{ij}$ and $H\cdot E_i=0$. Then, considering the coordinate functions
$x, y, z, s, t$ on $\mathbb{C}^5$ we have  
$$X_2=(\mathbb{C}^{\,5}-\mathcal{Z})\,/\,(\mathbb{C}^*)^3,$$
with $h_1=\deg(x)=H-E_1-E_2$, $h_2=\deg(y)=H-E_1$, $h_3=\deg(z)=H-E_2$, $h_4=\deg(s)=E_1$ and $h_5=\deg(t)=E_2$.\\

\noindent Consider now a holomorphic foliation $\mathcal{F}$ of degree 
$d=(d_0,d_1,d_2)=d_0H-d_1E_1-d_2E_2$ on $X_2$. Since $\texttt{C}_0(h)=1$, $\texttt{C}_1(h)=3H-E_1-E_2$ 
and $\texttt{C}_2(h)=5$, we have 
$$\#\,\mathrm{Sing}\left(\mathcal{F}\right)=\sum_{j=0}^{2}(-1)^j d^{\,2-j}\texttt{C}_j(h)
=d_0(d_0-3)+d_1(1-d_1)+d_2(1-d_2)+5.$$
\end{exe}

It is of interest to know that blow-ups of $\mathbb{P}^n$ in $0\leq r\leq n+1$ points in
general position are torics varieties, and that for $r > n+1$ these are not toric anymore.
Moreover, blow-ups of $\mathbb{P}^3$ along $k=1,2$ lines in general position are toric threefolds; see \cite{Du}. 

\begin{defi} Let $\mathbb{P}_{\Delta}$ be an $n$-dimensional compact toric orbifold with at most isolated singularities.
Let $\mathcal{D}$ be a codimension one distribution with isolated zeros and let $\omega$ be a holomorphic section of 
$\Omega^1_{\mathbb{P}_{\Delta}}(d)$ that induces $\mathcal{D}$. The index of $\mathcal{D}$ at a zero $p$ is defined by
$$\mathcal{I}_p^{orb}\left(\mathcal{D}\right)=\frac{1}{\#G_p}{Res}_{\,\tilde{p}}\left\{\frac{\mathrm{Det}(J\tilde{\omega})\,d\tilde{z}_1\wedge\dots\wedge d\tilde{z}_n}{\tilde{\omega}_1\dots\tilde{\omega}_n}\right\},$$
where $\pi_p:(\widetilde{U},\tilde{p})\rightarrow(U,p)$ denotes the local projection of $\mathbb{P}_{\Delta}$ at $p$:
$\tilde{\omega}= \pi_p^\ast \omega$,
$J{\tilde{\omega}}=\displaystyle\left(\frac{\partial\tilde{\omega}_i}{\partial
\tilde{z}_j}\right)_{1\leq i,\,j\leq n}$, 
$\mbox{Res}_{\,\tilde{p}}\displaystyle\left\{\frac{\mathrm{Det}\big(J\tilde{\omega}\,\big)\,d\tilde{z}_1\wedge\dots\wedge
d\tilde{z}_n}{\tilde{\omega}_1\dots\tilde{\omega}_n}\right\}$ is
Grothendieck's point residue and $(\tilde{z}_1,\dots,\tilde{z}_n)$ is a germ of  coordinate system on $(\widetilde{U},\tilde{p})$.
\end{defi}

This index was introduced by Satake \cite{Sata} for $C^{\infty}$-vector fields in the orbifold case, see also \cite{Ding, Mig}.
The Grothendieck point residue embodies the Poincar\'e-Hopf index, the Milnor number and the intersection number of $n$ divisors in
$\mathbb{C}^n$ which intersect properly, and has many uses in deep results such as the Baum-Bott theorem, which is a
generalization of both the Poincar\'e-Hopf theorem and the Gauss-Bonnet theorem in the complex realm, see \cite{Grif, So, Su}. 

\begin{prop} \label{exew1} Let $\mathbb{P}(\omega)$ be a well formed weighted projective space.
Let $\mathcal{D}$ be a codimension one distribution of degree $d$ with isolated zeros on $\mathbb{P}(\omega)$. Then
\begin{equation}\label{akz}
\sum_{p\in\mathrm{Sing}\left(\mathcal{D}\right)}\mathcal{I}_p^{orb}\left(\mathcal{D}\right)
=\frac{1}{\omega_{0}\cdots \omega_{n}}\sum_{j=0}^{n}(-1)^j \texttt{C}_j(\omega)\,d^{\,n-j}.
\end{equation}
\end{prop}
\begin{proof}
Consider the natural map $\varphi:\mathbb{P}^n\rightarrow\mathbb{P}(\omega)$, 
$\varphi\big(\left[z_0,\ldots,z_n\right]\big)=\big[z_0^{\omega_0},\ldots,z_n^{\omega_n}\big]_{\omega}$ of degree 
$\deg \varphi=\omega_0\cdots \omega_n$. It was shown in \cite{Ma} that there is a line orbifold bundle $\mathcal{O}_{\omega}(1)$ 
on $\mathbb{P}(\omega)$, unique up to isomorphism, such that $\varphi^{\ast}\mathcal{O}_{\omega}(1)=\mathcal{O}(1)$.
From the definition of orbifold integral we have
$$\int_{\mathbb{P}(\omega)}^{orb}c_{1}(\mathcal{O}_{\omega}(1))^{n}=
\frac{1}{\# Ker(\mathbb{P}(\omega))}\int_{{\mathbb{P}(\omega)}_{reg}}c_{1}(\mathcal{O}_{\omega}(1))^{n}.$$
Since $Ker(\mathbb{P}(\omega))=\bigcap_{i=0}^{n}\mu_{w_i}=\left\{1\right\}$, 
this implies $\# Ker(\mathbb{P}(\omega))=1$; see \cite{Ma}. Then
$$\int_{\mathbb{P}(\omega)}^{orb}c_{1}(\mathcal{O}_{\omega}(1))^{n}=\frac{1}{\deg\varphi}
\int_{\mathbb{P}^{n}}\varphi^{\ast}c_{1}(\mathcal{O}_{\omega}(1))^{n}=\frac{1}{\deg\varphi}
\int_{\mathbb{P}^{n}}c_{1}(\mathcal{O}(1))^{n}=\frac{1}{\omega_{0}\cdots \omega_{n}}.$$
Set $\textsc{H}=c_{1}(\mathcal{O}_{\omega}(1))$, $d = d\,\textsc{H}$ and $h_i=\deg(z_i)=\omega_i\,\textsc{H}$. 
Note that Proposition \ref{teo1} is still valid here, for more details see \cite{Bla, Mig, Grif, Iza, Jou}. In this case we have
\begin{eqnarray*} \label{equ}
\sum_{p\in\mathrm{Sing}\left(\mathcal{D}\right)}\mathcal{I}_p^{orb}\left(\mathcal{D}\right)
&=&\sum_{j=0}^{n}\left\{(-1)^j d^{\,n-j}\int_{\mathbb{P}(\omega)}^{orb} \texttt{C}_j
(\omega_1\,\textsc{H},\ldots,\omega_n\,\textsc{H})\cdot\textsc{H}^{n-j}\right\} \nonumber \\
&=&\sum_{j=0}^{n}\left\{(-1)^j d^{\,n-j}\texttt{C}_j(\omega)\int_{\mathbb{P}(\omega)}^{orb}\textsc{H}^{n}\right\} \nonumber \\
&=&\frac{1}{\omega_{0}\cdots \omega_{n}}\sum_{j=0}^{n}(-1)^j \texttt{C}_j(\omega)\,d^{\,n-j}.
\end{eqnarray*}
\end{proof}
\begin{exe} We provide three examples that offer a numerical illustration of the proposition mentioned earlier.
\begin{itemize}
	\item[(i)] Consider the codimension one holomorphic foliation $\mathcal{F}$ of degree $d=2m$ on $\mathbb{P}(1,1,m)$, 
defined by $w=mz_2z^{m-1}_0dz_0+mz_2z_1^{m-1}dz_1-(z_0^m+z_1^m)dz_2$. 
Set $\mu_m=\left\{\alpha\in\mathbb{C}^{\ast}:\alpha^m=1\right\}$. Then 
$$\mathrm{Sing}\left(\mathcal{F}\right)=\big\{\left[0:0:1\right]_{\omega}, 
\left[1:i\alpha:0\right]_{\omega}\,:\,\,\alpha\in\mu_m\big\}.$$
In coordinates $\big(U_0,\mu_m\big):z_0\neq0$, we have 
$w=mz_2z_1^{m-1}dz_1-(1+z_1^m)dz_2$ and $\mathcal{I}^{orb}_{(i\alpha,0)}=1$ for all $\alpha\in\mu_m$. 
In coordinates $\big(U_2,\left\{+1\right\}\big)\,:\,z_2\neq0$, we have $w=mz_0^{m-1}dz_0+mz_1^{m-1}dz_1$ and  
$\mathcal{I}^{orb}_{(0,0)}=\frac{(m-1)^2}{m}$. Therefore, we get
$$\sum_{p\in\mathrm{Sing}\left(\mathcal{F}\right)}\mathcal{I}_p^{orb}\left(\mathcal{F}\right)=m+\frac{(m-1)^2}{m}=
\frac{1}{1\cdot1\cdot m}\sum_{j=0}^{2}(-1)^j \texttt{C}_j(1,1,m)\,(2m)^{\,2-j}.$$
 \item[(ii)] Consider the codimension one holomorphic distribution $\mathcal{D}$ of degree $d=\omega_0+\omega_1=\omega_2+m\omega_3$ on 
$\mathbb{P}(\omega_0,\omega_1,\omega_2,\omega_3)$, 
defined by $w=-\omega_1z_1dz_0+\omega_0z_0dz_1-\omega_3z_3^mdz_2+\omega_2z_2z_3^{m-1}dz_3$. Then 
$$\mathrm{Sing}\left(\mathcal{D}\right)=\big\{\left[0:0:1:0\right]_{\omega}\big\}.$$
In coordinates $\big(U_2,\mu_{\omega_2}\big):z_2\neq0$, we have 
$w=-\omega_1z_1dz_0+\omega_0z_0dz_1+\omega_2z_3^{m-1}dz_3$
and $\mathcal{I}^{orb}_{(0,0,0)}=\frac{m-1}{\omega_2}$. On the other hand, writing $k=\omega_2+\omega_3$, we obtain 
$\texttt{C}_1(\omega)=d+k$, $\texttt{C}_2(\omega)=\omega_0\omega_1+\omega_2\omega_3+dk$ and 
$\texttt{C}_3(\omega)=\omega_0\omega_1k+\omega_2\omega_3d$.
Therefore, we have
$$\frac{1}{\omega_0\omega_1\omega_2\omega_3}\sum_{j=0}^{3}(-1)^j \texttt{C}_j(\omega)\,d^{\,3-j}
=\frac{\omega_0\omega_1(d-k)}{\omega_0\omega_1\omega_2\omega_3}=\frac{m-1}{\omega_2}
=\sum_{p\in\mathrm{Sing}\left(\mathcal{D}\right)}\mathcal{I}_p^{orb}\left(\mathcal{D}\right).$$
 \item[(iii)] Let $\mathcal{D}$ be a distribution of degree $d$ on $\mathbb{P}(1,1,1,k)$, $k>1$. 
Suppose $\mathcal{D}$ has only isolated singularities. If $k\nmid d^3-2$, then $\mathcal{D}$ has 
a singularity at $\mathrm{Sing}\big(\mathbb{P}(1,1,1,k)\big)=\left\{\bar{e}_3\right\}$.
In fact, suppose $\bar{e}_3$ is non-singular for $\mathcal{D}$. Then the right-hand side of $(\ref{akz})$ is
an integer and this happens only if $k\mid d^3-2$.
\end{itemize}
\end{exe}

\bigskip

Conditions for the existence of such natural maps $\varphi:\mathbb{P}^n\rightarrow\mathbb{P}_{\Delta}$ as in Proposition \ref{exew1}
are given in \cite{Cox4}. More generally, it is straightforward to see that the same reasoning as in the previous proposition can be 
extended here to yield:

\begin{prop} \label{teo20}
Let $\mathbb{P}_{\Delta}$ be an $n$-dimensional compact toric orbifold with only isolated singularities.
Let $\mathcal{D}$ be a codimension one distribution of degree 
$d$ with isolated zeros.
Suppose there is a natural map $\varphi:\mathbb{P}^n\rightarrow\mathbb{P}_{\Delta}$ of finite degree such that 
$\varphi^{\ast}\mathcal{O}_{\mathbb{P}_{\Delta}}(\left[D_i\right])=\mathcal{O}_{\mathbb{P}^n}(m_i)$.
Set $\varphi^{\ast}\mathcal{O}_{\mathbb{P}_{\Delta}}(d)=\mathcal{O}_{\mathbb{P}^n}(k)$.
Then
\begin{eqnarray*}
\sum_{p\in\mathrm{Sing}\left(\mathcal{D}\right)}\mathcal{I}_p^{orb}\left(\mathcal{D}\right)
=\frac{1}{\deg \varphi\cdot\# Ker(\mathbb{P}_{\Delta})}\sum_{j=0}^{n}(-1)^j \texttt{C}_j(m)\,k^{\,n-j},
\end{eqnarray*}
where $\texttt{C}_{j}(m)$ is the $j$th elementary symmetric 
function of the variables $m_1,\ldots, m_{n+r}$.
\end{prop}

\section{Applications}

We will say that a codimension one distribution $\mathcal{D}$ on $\mathbb{P}_{\Delta}$
is \emph{generic} if it has at most isolated singularities, 
and that $\mathcal{D}$ is \emph{regular} if $\mathrm{Sing}\left(\mathcal{D}\right)=\varnothing$.
For the rest of this section, we will make the assumption that all distributions are generic, and we will employ homogeneous coordinates.

\begin{cor} \label{az} Let $\mathbb{P}_{\Delta}$ be an $n$-dimensional compact toric manifold.
Set $h_i=\left[D_i\right]\in\mathcal{A}_{n-1}(\mathbb{P}_{\Delta})$.
Let $\mathcal{D}$ be a codimension one distribution of degree $d$.
Then, $\mathcal{D}$ is singular if $\gcd\left\{d_i\right\} \nmid \texttt{C}_n(h)$.
\end{cor}
\begin{proof}
By Proposition \ref{teo1}, if $\mathrm{Sing}\left(\mathcal{D}\right)=\varnothing$, then we can write
$$0=\sum_{j=0}^{n}\left\{(-1)^j\sum_{\left|k\right|=n-j}\binom{n-j}{k}
\int_{\mathbb{P}_{\Delta}} \texttt{C}_j(h)\cdot \prod_i \big(d_i h_i\big)^{k_i}\right\} 
=(-1)^n \texttt{C}_n(h) + \sum_{i}a_i d_i,$$
for some $a_i\in\mathbb{Z}$. Then $\gcd\left\{d_i\right\} \mid \texttt{C}_n(h)$.
\end{proof}

Likewise, the subsequent corollary can be derived from Proposition \ref{teo20} and Corollary \ref{az}:

\begin{cor} \label{orbcor} Let $\mathbb{P}_{\Delta}$ be an $n$-dimensional compact toric orbifold with at most isolated singularities. Let $\mathcal{D}$ be a codimension one distribution of degree $d$.
Suppose there is a natural map $\varphi:\mathbb{P}^n\rightarrow\mathbb{P}_{\Delta}$ of finite degree such that 
$\varphi^{\ast}\mathcal{O}_{\mathbb{P}_{\Delta}}(\left[D_i\right])=\mathcal{O}_{\mathbb{P}^m}(m_i)$ and 
$\varphi^{\ast}\mathcal{O}_{\mathbb{P}_{\Delta}}(d)=\mathcal{O}_{\mathbb{P}^m}(k)$ for some $k\neq 0$.
Then, $\mathcal{D}$ is regular if and only if 
$$\prod_{i=1}^{n+r}\left(k-m_i\right)=(-1)^n\sum_{i=1}^r(-1)^i\texttt{C}_{n+i}(m)\,k^{\,r-i}.$$
Moreover, $\mathcal{D}$ is singular if $k \nmid \texttt{C}_n(m)$.
\end{cor}
\begin{proof} By Proposition \ref{teo20} we have, $\mathcal{D}$ is regular if and only if
$$\sum_{j=0}^{n}(-1)^j \texttt{C}_j(m)\,k^{\,n-j}=0.$$
Then, $\mathcal{D}$ is regular if and only if
\begin{eqnarray*}
\prod_{i=1}^{n+r}\left(k-m_i\right)&=&\sum_{j=0}^{n+r}(-1)^j \texttt{C}_j(m)\,k^{\,n+r-j}\\
&=& k^{\,r}\sum_{j=0}^{n}(-1)^j \texttt{C}_j(m)\,k^{\,n-j} + \sum_{j=n+1}^{n+r}(-1)^j \texttt{C}_j(m)\,k^{\,n+r-j}\\
&=&(-1)^n\sum_{i=1}^r(-1)^i\texttt{C}_{n+i}(m)\,k^{\,r-i}.
\end{eqnarray*}
\end{proof}

\begin{exe} Consider the multiprojective space $\mathbb{P}_{\Delta}=\mathbb{P}^n\times\mathbb{P}^m$.
Let $\mathcal{D}$ be a codimension one distribution of degree $d=(d_1, d_2)$
on $\mathbb{P}^n\times\mathbb{P}^m$. Then, 
$\mathcal{D}$ is singular if $\gcd \left(d_1,d_2\right) \nmid (n+1)(m+1)$, see Example \ref{exmul}.
Similarly, it is easy to see that on $\mathbb{P}_{\Delta}=\mathbb{P}^{n_1}\times\cdots\times\mathbb{P}^{n_k}$, a 
codimension one distribution of degree $d=(d_1,\ldots, d_k)$ is singular if
$\gcd\left\{d_i\right\} \nmid \prod_{i=1}^k (n_i+1)$.
\end{exe}

Now, let's turn our attention to low-dimensional multiprojective spaces:

\begin{cor}\label{Me} Let $\mathcal{D}$ be a codimension one distribution on $\mathbb{P}_{\Delta}$.
\begin{enumerate}
	\item If $\mathbb{P}_{\Delta}=\mathbb{P}^{1}\times\mathbb{P}^{1}\times\mathbb{P}^{1}$ and $\mathcal{D}$ regular, we have
$\deg\hspace{-0.05cm}\left(\mathcal{D}\right)\in\left\{(2,0,0), (0,2,0), (0,0,2)\right\}$. Then, in homogeneous coordinates, $\mathcal{D}$ 
is induced by
$$\omega=z_{i1}dz_{i0}-z_{i0}dz_{i1},\,\,\mbox{for some}\,\,i=1,2,3.$$
In particular $\mathcal{D}$ is integrable and admits a rational first integral.
Moreover $\mathcal{D}$ is a natural fibration given by $\pi_i:\mathbb{P}_{\Delta}\rightarrow\mathbb{P}^1$, for some $i=1,2,3$.
\vspace{0.1cm}
	\item If $\mathbb{P}_{\Delta}=\mathbb{P}^{2}\times\mathbb{P}^{1}$ and $\mathcal{D}$ regular, we have
$\deg\hspace{-0.05cm}\left(\mathcal{D}\right)=(0,2)$. Then, in homogeneous coordinates, $\mathcal{D}$ is induced by
$$\omega=z_{21}dz_{20}-z_{20}dz_{21}.$$
In particular $\mathcal{D}$ is integrable and admits a rational first integral.
Moreover $\mathcal{D}$ is the natural fibration given by $\pi_2:\mathbb{P}_{\Delta}\rightarrow\mathbb{P}^1$.
\vspace{0.1cm}
	\item If $\mathbb{P}_{\Delta}=\mathbb{P}^{2}\times\mathbb{P}^{2}$, we have $\mathrm{Sing}(\mathcal{D})\neq\varnothing$,
that is, there are no regular distributions. 
\end{enumerate}
\end{cor}
\begin{proof}
$(1)$ Consider $\mathbb{P}_{\Delta}=\mathbb{P}^{1}\times\mathbb{P}^{1}\times\mathbb{P}^{1}$. 
Here $h_1=h_{10}=h_{11}$, $h_2=h_{20}=h_{21}$ and $h_3=h_{30}=h_{31}$, with $h_1^2=h_2^2=0$ and $h_1\cdot h_2=1$. 
Let $\mathcal{D}$ be a codimension one regular distribution of degree $d=(d_1,d_2,d_3)=d_1 h_1+d_2h_2+d_3h_3$
on $\mathbb{P}_{\Delta}$. By Proposition \ref{teo1} we have
$$\sum_{j=0}^{3}(-1)^j \texttt{C}_j(h)\left(d_1 h_1+d_2h_2+d_3h_3\right)^{3-j}=0.$$
Substituting $\texttt{C}_0(h)=1$, $\texttt{C}_1(h)=2\left(h_1+h_2+h_3\right)$, $\texttt{C}_2(h)=4\left(h_1h_2+h_1h_3+h_2h_3\right)$ 
and $\texttt{C}_3(h)=8\,h_1h_2h_3$ we obtain
$$3\,d_1d_2d_3-2\left(d_1d_2+d_1d_3+d_2d_3\right)+2\left(d_1+d_2+d_3\right)-4=0.$$ 
Since at least one $d_i\geq 1$, the integer solutions are given by $d=\left(1,1,2\right)$, $d=\left(1,-1,-2\right)$ and $d=\left(2,0,0\right)$
(and permutation of $d_i$). Using homogeneous coordinates, $\mathcal{D}$ is induced by a polynomial form 
$$\omega=\sum_{i=1}^3 \left(P_{i0}dz_{i0}+P_{i1}dz_{i1}\right),$$ 
where $\deg(P_{10})=\deg(P_{11})=(d_1-1,d_2,d_3)$, $\deg(P_{20})=\deg(P_{21})=(d_1,d_2-1,d_3)$ and $\deg(P_{30})=\deg(P_{31})=(d_1,d_2,d_3-1)$.
Also, $i_{R_i} \, \omega=0$ for $R_i = z_{i0} \frac{\partial}{\partial z_{i0}}+z_{i1} \frac{\partial}{\partial z_{i1}}$, $i=1,2,3$.
So, it is easy to check that $d=\left(2,0,0\right)$ (and permutation of $d_i$), and so $\mathcal{D}$ is induced by
$$\omega=z_{i1}dz_{i0}-z_{i0}dz_{i1},\,\,\mbox{for some}\,\,i=1,2,3.$$
$(2)$ Consider $\mathbb{P}_{\Delta}=\mathbb{P}^{2}\times\mathbb{P}^{1}$. 
Here $h_1=h_{10}=h_{11}=h_{12}$ and $h_2=h_{20}=h_{21}$, with $h_1^3=h_2^2=0$ and $h_1^2\cdot h_2=1$. 
Let $\mathcal{D}$ be a codimension one regular distribution of degree $d=(d_1,d_2)=d_1 h_1+d_2h_2$
on $\mathbb{P}_{\Delta}$. By Proposition \ref{teo1} we have
$$\sum_{j=0}^{3}(-1)^j \texttt{C}_j(h)\left(d_1 h_1+d_2h_2\right)^{3-j}=0.$$
Substituting $\texttt{C}_0(h)=1$, $\texttt{C}_1(h)=3h_1+2h_2$, $\texttt{C}_2(h)=3h_1^2+6h_1h_2$ 
and $\texttt{C}_3(h)=6\,h_1^2h_2$ we obtain
$$d_2=\frac{2\left(d_1^2-3d_1+3\right)}{3\left(d_1-1\right)^2}.$$
Thus, it is easy to check that $d=\left(0,2\right)$, and so $\mathcal{D}$ is induced by
$$\omega=z_{21}dz_{20}-z_{20}dz_{21}.$$
$(3)$ Consider $\mathbb{P}_{\Delta}=\mathbb{P}^{2}\times\mathbb{P}^{2}$. 
Here $h_1=h_{10}=h_{11}=h_{12}$ and $h_2=h_{20}=h_{21}=h_{22}$, with $h_1^3=h_2^3=0$ and $h_1^2\cdot h_2^2=1$. 
Let $\mathcal{D}$ be a codimension one distribution of degree $d=(d_1,d_2)=d_1 h_1+d_2h_2$
on $\mathbb{P}_{\Delta}$. By Proposition \ref{teo1} we have
$$\#\,\mathrm{Sing}\hspace{-0.05cm}\left(\mathcal{D}\right)=\sum_{j=0}^{4}(-1)^j \texttt{C}_j(h)\left(d_1 h_1+d_2h_2\right)^{4-j}.$$
Substituting $\texttt{C}_0(h)=1$, $\texttt{C}_1(h)=3h_1+3h_2$, $\texttt{C}_2(h)=3h_1^2+9h_1h_2+3h_2^2$, 
$\texttt{C}_3(h)=9h_1^2h_2+9h_1h_2^2$, and $\texttt{C}_4(h)=9h_1^2h_2^2$ we obtain
$$\#\,\mathrm{Sing}\hspace{-0.05cm}\left(\mathcal{D}\right)=2 d_1^2d_2^2-3d_1d_2^2-3d_1^2d_2+d_2^2+6d_1d_2+d_1^2-3d_2-3d_1+3.$$ 
Hence, it is easy to check that $\#\,\mathrm{Sing}\hspace{-0.05cm}\left(\mathcal{D}\right)\neq 0$ for all $d\in \mathbb{Z}\times\mathbb{Z}$.
\end{proof}

\bigskip

\begin{exe} Consider the Hirzebruch surface $\mathcal{H}_r$, $r\geq 0$.
Here $\mathrm{Pic}(\mathcal{H}_r)=\mathbb{Z}h_1\oplus\mathbb{Z}h_2$, where 
$h_1^2=0$, $h_1\cdot h_2=1$ and $h_2^2=-r$. Moreover $h_3=h_1$ and $h_4=rh_1+h_2$; see \cite{Cox2}.
In particular $\texttt{C}_2(h)=\sum_{1\leq i<j\leq 4}h_i\cdot h_j=4$. 
Let $\mathcal{F}$ be a holomorphic foliation of degree $d=(d_1, d_2)$ on $\mathcal{H}_r$. 
Then, $\mathcal{F}$ is singular if $\gcd \left(d_1,d_2\right) \nmid 4$. \\ 

In \cite[p.37]{Bru}, M. Brunella shows that if $\mathcal{F}$ is a regular foliation on $\mathcal{H}_r$, then
$\mathcal{F}$ is a $\mathbb{P}^1$-fibration over $\mathbb{P}^1$.  
We will now present an alternate proof for this assertion. Utilizing homogeneous coordinates, 
we shall provide normal forms for these foliations:

\end{exe}
\begin{cor} \label{B}
Let $\mathcal{F}$ be a holomorphic foliation of degree $d=(d_1, d_2)$ on $\mathcal{H}_r$. 
Then, $\mathcal{F}$ is regular if and only if
\begin{itemize}
	\item[(1)] $d_1=0$ and $d_2=2$, for $r=0$. Then, in homogeneous coordinates, $\mathcal{F}$ is induced by
$$\omega=z_{2,2}dz_{1,2}-z_{1,2}dz_{2,2};$$
or
	\item[(2)] $d_1=2$ and $d_2=0$, for $r\geq0$. Then, in homogeneous coordinates, $\mathcal{F}$ is induced by
$$\omega=z_{2,1}dz_{1,1}-z_{1,1}dz_{2,1}.$$
\end{itemize}
In particular $\mathcal{F}$ is a $\mathbb{P}^1$-fibration over $\mathbb{P}^1$;
$\mathcal{F}$ is given by the canonical projection $\mathcal{H}_r\rightarrow\mathbb{P}^1$. \\
Moreover, there are no holomorphic foliations with a unique singularity (with multiplicity $1$).
\end{cor}
\begin{proof}
Set $d=d_1h_1+d_2h_2$. By Proposition \ref{teo1} we have,
$\mathcal{F}$ is regular if and only if
\begin{eqnarray*}
0 &=&\sum_{j=0}^{2}(-1)^j d^{\,2-j}\texttt{C}_j(h) \\
&=& d^{\,2}-d\,\texttt{C}_1(h)+\texttt{C}_2(h) \\
&=& 2d_1d_2-d_2^{\,2}r-\left(d_1h_1+d_2h_2\right)\big((2+r)h_1+2h_2\big)+4 \\
&=& 2d_1d_2-2d_1-2d_2+4-d_2\left(d_2-1\right)r \\
&=& 2(d_1-1)(d_2-1)+2-d_2\left(d_2-1\right)r. 
\end{eqnarray*}
Then, $\mathcal{F}$ is regular if and only if $(d_2-1)\big(d_2r-2(d_1-1)\big)=2$. Then $d_2=0$, $d_2=2$ or $d_2=3$.
If $d_2=0$, then $d_1=2$. If $d_2=2$, then $d_1=r$. If $d_2=3$, then we have $2d_1-3r=1$.
Now, let us show $d_1=r>0$ and $d_2=2$ is not possible. Suppose $d_1=r>0$ and $d_2=2$, then 
$\mathcal{F}$ is induced by a polynomial form 
$\omega=\sum P_{i,j}dz_{i,j}$, where 
$\deg(P_{1,1})=\deg(P_{2,1})=(d_1-1,d_2)=(r-1,2)$, $\deg(P_{1,2})=(d_1,d_2-1)=(r,1)$ and $\deg(P_{2,2})=(d_1-r,d_2-1)=(0,1)$. 
Then, the polynomials $P_{i,j}$ are of the form: $P_{1,1}=Az_{1,2}^2$, $P_{2,1}=Bz_{1,2}^2$, 
$P_{1,2}=Cz_{1,2}+\alpha z_{2,2}$ and $P_{2,2}=\beta z_{1,2}$, 
for certain polynomials $A,B,C \in \mathbb{C}\left[z_{1,1},z_{2,1}\right]$ and $\alpha,\beta\in \mathbb{C}$. 
Since $i_{R_2}\omega=0$, we have
$$0=z_{1,2}P_{1,2}+z_{2,2}P_{2,2}=Cz_{1,2}^2+\alpha z_{1,2}z_{2,2}+\beta z_{2,2}z_{1,2},$$
which implies that $\alpha+\beta=0$ and $C=0$. On the other hand, since $i_{R_1}\omega=0$, we have
$$0=z_{1,1}P_{1,1}+z_{2,1}P_{2,1}+rz_{2,2}P_{2,2}=z_{1,1}Az_{1,2}^2+z_{2,1}Bz_{1,2}^2+r\beta z_{2,2}z_{1,2},$$
which implies that $\beta=0$. Therefore $\mathcal{F}$ 
is induced by 
$$\omega=P_{1,1}dz_{1,1}+P_{2,1}dz_{2,1}=z_{1,2}^2\big(Adz_{1,1}+Bdz_{2,1}\big),$$
that contradicts the fact that $\mathcal{F}$ is regular. Now, let us show $d_2=3$ and $2d_1-3r=1$ is not possible.
Suppose $d_2=3$ and $2d_1-3r=1$, then
$\mathcal{F}$ is induced by a polynomial form 
$\omega=\sum P_{i,j}dz_{i,j}$, where 
$\deg(P_{1,1})=\deg(P_{2,1})=(d_1-1,d_2)=(\frac{3r-1}{2},3)$, $\deg(P_{1,2})=(d_1,d_2-1)=(\frac{3r+1}{2},2)$ 
and $\deg(P_{2,2})=(d_1-r,d_2-1)=(\frac{r+1}{2},2)$. 
Then, the polynomials $P_{i,j}$ are of the form: 
$P_{1,1}=A_1z_{1,2}^3+A_2z_{1,2}^2z_{2,2}$, 
$P_{2,1}=B_1z_{1,2}^3+B_2z_{1,2}^2z_{2,2}$, 
$P_{1,2}=C_1z_{1,2}^2+C_2z_{1,2}z_{2,2}+\alpha z_{2,2}^2$ and 
$P_{2,2}=Dz_{1,2}^2+\beta z_{1,2}z_{2,2}$, 
for certain polynomials $A_1,A_2,B_1,B_2,C_1,C_2,D \in \mathbb{C}\left[z_{1,1},z_{2,1}\right]$ and 
$\alpha,\beta\in \mathbb{C}$, with $\alpha=\beta=0$ if $r\neq1$. 
Since $i_{R_1}\omega=0$, we have
$z_{1,1}P_{1,1}+z_{2,1}P_{2,1}+rz_{2,2}P_{2,2}=0$, which implies that $\beta=0$. 
On the other hand, since $i_{R_2}\omega=0$, we have
$z_{1,2}P_{1,2}+z_{2,2}P_{2,2}=0$, which implies $\alpha+\beta=0$. 
Therefore $\mathcal{F}$ is induced by
$$\omega=z_{1,2}\,\omega_0,$$
that contradicts the fact that $\mathcal{F}$ is regular. \\
Let us see the case $d_1=r=0$ and $d_2=2$. Then $\mathcal{F}$ is induced by a polynomial form 
$\omega=\sum P_{i,j}dz_{i,j}$ with 
$\deg(P_{1,1})=\deg(P_{2,1})=(d_1-1,d_2)=(-1,2)$, $\deg(P_{1,2})=(d_1,d_2-1)=(0,1)$ and $\deg(P_{2,2})=(d_1-r,d_2-1)=(0,1)$. 
Then $\omega=\left(a_{1,2}z_{1,2}+a_{2,2}z_{2,2}\right)dz_{1,2}+\left(b_{1,2}z_{1,2}+b_{2,2}z_{2,2}\right)dz_{2,2}$. 
Since $i_{R_2}\omega=0$, we have $a_{1,2}=b_{2,2}=0$ and $a_{2,2}+b_{1,2}=0$. Therefore $\mathcal{F}$ 
is induced by
$$\omega=k\left(z_{2,2}dz_{1,2}-z_{1,2}dz_{2,2}\right),\,\,\,k\in\mathbb{C}^{\ast}.$$
Similarly, if $d_1=2$ and $d_2=0$, then $\mathcal{F}$ is induced by a polynomial form 
$\omega=\sum P_{i,j}dz_{i,j}$ with 
$\deg(P_{1,1})=\deg(P_{2,1})=(d_1-1,d_2)=(1,0)$, $\deg(P_{1,2})=(d_1,d_2-1)=(2,-1)$ and $\deg(P_{2,2})=(d_1-r,d_2-1)=(2-r,-1)$. 
Then $\omega=\left(a_{1,1}z_{1,1}+a_{2,1}z_{2,1}\right)dz_{1,1}+\left(b_{1,1}z_{1,1}+b_{2,1}z_{2,1}\right)dz_{2,1}$. 
Since $i_{R_1}\omega=0$, we have $a_{1,1}=b_{2,1}=0$ and $a_{2,1}+b_{1,1}=0$. Therefore $\mathcal{F}$ 
is induced by
$$\omega=k\left(z_{2,1}dz_{1,1}-z_{1,1}dz_{2,1}\right),\,\,\,k\in\mathbb{C}^{\ast}.$$
Finally, $\mathcal{F}$ has a unique singularity if and only if $(d_2-1)\big(d_2r-2(d_1-1)\big)=-1$. 
Now, it is easy to see that this last equation has no integer solutions.
\end{proof}

\begin{exe} Consider the $n$-dimensional rational normal scroll $\mathbb{P}_{\Delta}=\mathbb{F}(a)=
\mathbb{F}(a_1,\ldots,a_n)$, $a_1,\ldots,a_n\in\mathbb{Z}$.
Here $\mathrm{Pic}\big(\mathbb{F}(a)\big)=\mathbb{Z}L\oplus\mathbb{Z}M$, 
subject to the relations $L^2=0$, $M^n=\left|a\right|=\sum_{i=1}^n a_i$ and $M^{n-1}\cdot L=1$. Moreover
$h_1=h_2=L$ and $h_i'=h_{i+2}=-a_iL+M$ for all $i=1,\ldots,n$; see \cite{Re}.
In particular we have 
\begin{eqnarray*}
\texttt{C}_n(h)&=&\texttt{C}_n\left(h_1,h_2,h_1',\ldots,h_n'\right) \\
&=& h_1' \cdots h_n' + 2\sum_{i_1<\cdots<i_{n-1}} L\cdot h_{i_1}' \cdots h_{i_{n-1}}' \\
&=& M^n - \left(\sum_{i=1}^n a_i\right) M^{n-1} \cdot L + 2 \sum_{i_1<\cdots<i_{n-1}} L\cdot M^{n-1} \\
&=& 2n.
\end{eqnarray*}

Let $\mathcal{D}$ be a codimension one distribution of degree $d=(d_1, d_2)$ on $\mathbb{F}(a)$. 
Then, $\mathcal{D}$ is singular if $\gcd \left(d_1,d_2\right) \nmid 2n$. Noting that $\mathbb{F}(-a,0)\simeq\mathcal{H}_a$.
More generally we prove the following theorem:
\end{exe}

\begin{teo} \label{A} Let $\mathcal{D}$ be a codimension one distribution of degree $d=(d_1, d_2)$ on 
$\mathbb{F}(a)=\mathbb{F}(a_1,\ldots,a_n)$, $n>2$.
Then, $\mathcal{D}$ is regular if and only if $d_1=2$ and $d_2=0$. Then, in homogeneous coordinates, $\mathcal{D}$ is induced by
$$\omega=z_{1,2}dz_{1,1}-z_{1,1}dz_{1,2}.$$
In particular $\mathcal{D}$ is integrable and admits a rational first integral.
Moreover $\mathcal{D}$ is the natural fibration given by $\mathbb{F}(a)\rightarrow\mathbb{P}^1$.
\end{teo}
\begin{proof} Set $d=d_1L+d_2M$ and put $\left|a\right|=\sum a_i$. Then, $\mathcal{D}$ is regular if and only if
\begin{eqnarray} \label{scro}
\sum_{i=0}^{n}(-1)^i d^{\,n-i}\texttt{C}_i(h)=0.
\end{eqnarray}
On the other hand, for $m>0$ we have 
$$d^{\,m}=\sum_{i=0}^m\binom{m}{i}(d_1L)^{m-i}(d_2M)^i=md_1d_2^{\,m-1}LM^{m-1}+d_2^{\,m}M^m,$$
and 
\begin{eqnarray*}
\texttt{C}_m(h) &=& 2\sum_{i_1<\cdots<i_{m-1}}L\prod_{k=1}^{m-1}(-a_{i_k}L+M)+
\sum_{i_1<\cdots<i_m}\prod_{k=1}^{m}(-a_{i_1}L+M) \\
&=& 2\sum_{i_1<\cdots<i_{m-1}}LM^{m-1}+\sum_{i_1<\cdots<i_m}\Big(-\left(\Sigma_{k=1}^ma_k\right)LM^{m-1}+M^m\Big) \\
&=& 2\binom{n}{m-1} LM^{m-1}-\Big(\sum_{i_1<\cdots<i_m}\left(\Sigma_{k=1}^ma_k\right)\Big) LM^{m-1}+\binom{n}{m} M^m \\
&=& 2\binom{n}{m-1} LM^{m-1}- \left|a\right|\binom{n-1}{m-1} LM^{m-1}+\binom{n}{m} M^m \\
&=& \left(2\binom{n}{m-1} - \left|a\right|\binom{n-1}{m-1}\right) LM^{m-1}+\binom{n}{m} M^m.
\end{eqnarray*}


Then, for $0<i<n$, we have
$$d^{\,n-i}\texttt{C}_i(h) = \left(n-i\right)\binom{n}{i}d_1d_2^{\,n-i-1}+\left|a\right|\binom{n}{i}d_2^{\,n-i}+
\left(2\binom{n}{i-1} - \left|a\right|\binom{n-1}{i-1}\right)d_2^{\,n-i}.$$

Substituting into (\ref{scro}) we obtain 
\begin{eqnarray*}
0&=&\left(nd_1d_2^{\,n-1}+\left|a\right|d_2^{\,n}\right)+\sum_{i=1}^{n-1}(-1)^id^{\,n-i}\texttt{C}_i(h)+(-1)^n\texttt{C}_n(h) \\
&=& nd_1d_2^{\,n-1}+\left|a\right|d_2^{\,n}+\sum_{i=1}^{n-1}(-1)^i\left(n-i\right)\binom{n}{i}d_1d_2^{\,n-i-1}
+\left|a\right|\sum_{i=1}^{n-1}(-1)^i\binom{n}{i}d_2^{\,n-i} \\
&&+\sum_{i=1}^{n-1}(-1)^i\left(2\binom{n}{i-1} - \left|a\right|\binom{n-1}{i-1}\right)d_2^{\,n-i}+(-1)^n2n \\
&=& d_1\sum_{i=0}^{n-1}(-1)^i\left(n-i\right)\binom{n}{i}d_2^{\,n-i-1}
-2\sum_{i=1}^{n-1}(-1)^{i-1}\binom{n}{i-1}d_2^{\,n-i}+(-1)^n 2n \\
&& +\left|a\right| \left(d_2^{\,n}+\sum_{i=1}^{n-1}(-1)^i\left(\binom{n}{i}-\binom{n-1}{i-1}\right)d_2^{\,n-i}\right) \\
&=& nd_1\sum_{i=0}^{n-1}(-1)^i \binom{n-1}{i}d_2^{\,n-1-i}
-2\left(\,\sum_{i=1}^{n-1}(-1)^{i-1}\binom{n}{i-1}d_2^{\,n-i}+(-1)^n(1-n)\right)+2(-1)^n \\
&& +\left|a\right| d_2\left(d_2^{\,n-1}+\sum_{i=1}^{n-1}(-1)^i\binom{n-1}{i}d_2^{\,n-1-i}\right) \\
&=& nd_1(d_2-1)^{n-1}
-2P(d_2)+2(-1)^n +\left|a\right| d_2 (d_2-1)^{n-1},
\end{eqnarray*}  


\noindent where $P(t)=\sum_{i=1}^{n-1}(-1)^{i-1}\binom{n}{i-1}t^{\,n-i}+(-1)^n(1-n)=
t\sum_{i=0}^{n-2}(-1)^{i}\binom{n}{i}t^{\,n-2-i}+(-1)^n(1-n)$. 
Note that $P(1)=\sum_{i=0}^{n-2}(-1)^{i}\binom{n}{i}+(-1)^n(1-n)=\sum_{i=0}^{n}(-1)^{i}\binom{n}{i}=0$, 
so there exists a polynomial $Q(t)\in\mathbb{Z}\left[t\right]$ such that $P(t)=(t-1)Q(t)$.
Then we have, $\mathcal{D}$ is regular if and only if 
\begin{eqnarray} \label{dio}
(d_2-1)\Big((nd_1+\left|a\right|d_2)(d_2-1)^{\,n-2}-2Q(d_2)\Big)=2(-1)^{n+1}.
\end{eqnarray}
Therefore $d_2\in\left\{-1, 0, 2, 3\right\}$. Note that $d_2\neq -1, 3$ because $n>2$ in (\ref{dio}). 
If $d_2=0$, substituting into (\ref{dio}) we have
$$d_1=\frac{2(-1)^n+2Q(0)}{n(-1)^{n-2}}=\frac{2(-1)^n-2(-1)^n(1-n)}{n(-1)^{n-2}}=2.$$
 
\noindent Finally, when $d_2=2$, upon substituting into equation (\ref{dio}), we obtain
\begin{eqnarray*}
d_1&=&\frac{2P(2)+2(-1)^{n+1}-2\left|a\right|}{n} \\
&=& \frac{2}{n}\big(P(2)-(-1)^n-\left|a\right|\big) \\
&=&\frac{2}{n}\left(2\sum_{i=0}^{n-2}(-1)^{i}\binom{n}{i}2^{\,n-2-i}-n(-1)^n -\left|a\right|\right). \\
&=&\frac{1}{n}\left(\,\sum_{i=0}^{n-2}(-1)^{i}\binom{n}{i}2^{\,n-i}-2n(-1)^n -2\left|a\right|\right). \\
&=&\frac{1}{n}\left(\,\sum_{i=0}^{n}(-1)^{i}\binom{n}{i}2^{\,n-i}
-(-1)^{n-1}\binom{n}{n-1}2^1-(-1)^n\binom{n}{n}2^0-2n(-1)^n-2\left|a\right|\right)\\
&=&\frac{1}{n}\big(\,(2-1)^n+2n(-1)^n-(-1)^n-2n(-1)^n -2\left|a\right|\big)\\
&=&\frac{1}{n}\big(1+(-1)^{n+1}-2\left|a\right|\big).
\end{eqnarray*} 
Suppose $d_2=2$ and $d_1=\frac{1}{n}\big(1+(-1)^{n+1}-2\left|a\right|\big)$. Then 
$\mathcal{D}$ is induced by a polynomial form 
$\omega=\sum P_{i,j}dz_{i,j}$, where 
$\deg(P_{1,1})=\deg(P_{1,2})=(d_1-1,d_2)=(d_1-1,2)$ and $\deg(P_{2,i})=(d_1+a_i,d_2-1)=(d_1+a_i,1)$. 
Then, the polynomials $P_{i,j}$ are of the form: 
$$P_{1,1}=\sum_{q_i+q_j=2}A_{ij}\,z_{2,i}^{q_i}z_{2,j}^{q_j},\,\,P_{1,2}=\sum_{q_i+q_j=2}B_{ij}\,z_{2,i}^{q_i}z_{2,j}^{q_j},\,\,
\mbox{and}\,\,P_{2,i}=\sum_{i=1}^nC_i\,z_{2,i},$$
for certain polynomials $A_{ij},B_{ij},C_i \in \mathbb{C}\left[z_{1,1},z_{1,2}\right]$. 
Since $i_{R_2}\omega=0$, we have
$0=\sum_{i=1}^n z_{2,i}P_{2,i}=\sum_{i=1}^n C_iz_{2,i}^2$, 
which implies that $C_i=0$. On the other hand, since $i_{R_1}\omega=0$, we have
$z_{1,1}P_{1,1}+z_{1,2}P_{1,2}=0$, which implies that $P_{1,1}=z_{1,2}Q$ and $P_{1,2}=-z_{1,1}Q$, for some non-constant polynomial $Q$. 
Therefore $\mathcal{D}$ is induced by
$$\omega=Q\,\omega_0,$$
that contradicts the fact that $\mathcal{D}$ is regular. \\
Let us see the case $d_1=2$ and $d_2=0$. 
Then $\mathcal{D}$ is induced by a polynomial form 
$\omega=\sum P_{i,j}dz_{i,j}$ with 
$\deg(P_{1,1})=\deg(P_{1,2})=(1,0)$ and $\deg(P_{2,i})=(2+a_i,-1)$ for all $i=0,\ldots,n$.
Then $\omega=\left(a_{1,1}z_{1,1}+a_{1,2}z_{1,2}\right)dz_{1,1}+\left(b_{1,1}z_{1,1}+b_{1,2}z_{1,2}\right)dz_{1,2}$. 
Since $i_{R_1}\omega=0$, we have $a_{1,1}=b_{1,2}=0$, and $a_{1,2}+b_{1,1}=0$. Therefore $\mathcal{D}$ 
is induced by
$$\omega=k\left(z_{1,2}dz_{1,1}-z_{1,1}dz_{1,2}\right),\,\,\,k\in\mathbb{C}^{\ast}.$$
Note that $\mathrm{Sing}(\omega)=Z(z_{1,1},z_{1,2})\subset\mathcal{Z}$.
\end{proof}

The following theorem characterizes all regular distributions on a well formed weighted projective space:

\begin{teo} \label{wei} Let $\mathbb{P}(\omega)$ be a well formed $n$-dimensional weighted projective space.
Let $\mathcal{D}$ be a codimension one distribution of degree $d$ on $\mathbb{P}(\omega)$. Then 
$\mathcal{D}$ is regular if and only if $n$ is odd and 
$$\prod_{i=0}^n\omega_i=\prod_{i=0}^{n}\left(d-\omega_i\right).$$
\end{teo}
\begin{proof} By Corollary \ref{orbcor} $(r=1)$ we have, $\mathcal{D}$ is regular if and only if
$$\prod_{i=0}^{n}\left(d-\omega_i\right)=(-1)^{n+1}\texttt{C}_{n+1}(\omega).$$
Suppose $n$ is even, then there exists $i_0$ such that $d-\omega_{i_0}<0$.
This is a contradiction since in homogeneous coordinates, $\mathcal{D}$ is induced by a polynomial form 
$\Omega=\sum_{i}P_{i}dz_i$ with $\deg(P_i)=d-\omega_i$ for all $i=0,\ldots,n$, and so $n$ is odd.
\end{proof}

Finally, let us utilize Theorem \ref{wei} to construct a family of examples: 

\begin{exe} \label{expr} Let us consider the $(2m+1)$-dimensional well formed weighted projective space
$\mathbb{P}(w)=\mathbb{P}(w_0,w_1,\ldots,w_{2m},w_{2m+1})$. Let $\sigma\in \textsc{S}_{2m+2}$ be a permutation such that 
$\sigma^2=1$ and $\sigma(i)\neq i$ for all $i$.
Suppose $\omega_i+\omega_{\sigma(i)}=d$ is constant for all $i$. Now, in homogeneous coordinates, consider the $1$-form 
$$\Omega_{\sigma}=\sum_{i=0}^m c_i \Big(\omega_{i}z_{i}dz_{\sigma(i)} - 
\omega_{\sigma(i)}z_{\sigma(i)}dz_{i}\Big),\,\,c_i\in\mathbb{C}^{\ast}.$$
Note that $i_R\,\Omega_{\sigma}=0$, where $R=\sum_i\omega_iz_i\frac{\partial}{\partial z_i}$. Then $\Omega_{\sigma}$ 
induces a non-integrable regular distribution of degree $d$. Moreover, if we denote by $K_{\mathbb{P}(w)}$ the canonical divisor 
of $\mathbb{P}(w)$, then we conclude that
$$\mathcal{O}_{\mathbb{P}(w)}\hspace{-0.07cm}\left(-K_{\mathbb{P}(w)}\right)=\mathcal{O}_{\mathbb{P}(w)}\big(\Sigma w_i\big)
=\mathcal{O}_{\mathbb{P}(w)}(d)^{\otimes(m+1)},$$
that is, the $1$-form $\Omega_{\sigma}$ is a contact form. 
\end{exe}

\begin{exe}
The well formed weighted projective space $\mathbb{P}(1,w_1,w_2,w_3)$ does not admit a regular distribution
if $w_1\geq 2\max\left\{w_2, w_3\right\}$. In fact, suppose that there exists a regular distribution of degree $d$, 
then by Theorem \ref{wei} we have $(d-1)(d-w_1)(d-w_2)(d-w_3)=w_1w_2w_3$, which is a contradiction because 
the equation $(x-1)(x-w_1)(x-w_2)(x-w_3)=w_1w_2w_3$ does not admit integer solution for $x>w_1$; the polynomial
$p(x)=(x-1)(x-w_1)(x-w_2)(x-w_3)-w_1w_2w_3$ has a zero in $\left\langle w_1,w_1+1\right\rangle$ and $p'(x)>0$ for $x>w_1$.
Moreover, note that for $w_1=w_2+w_3-1< 2\max\left\{w_2, w_3\right\}$ there are regular distributions of degree $d=w_1+1=w_2+w_3$.
More generally, the weighted projective space $\mathbb{P}(1,w_1,w_2,\ldots,w_{2m+1})$ does not admit a 
regular distribution if $w_{i_0}\geq 2\max_{\,i\neq i_0}\left\{w_i\right\}$ for some $i_0$.
\end{exe}

\bigskip

\noindent{\footnotesize \textsc{Acknowlegments.}
I would like to be thankful to Maur\'icio Corr\^ea for interesting conversations.
I also thank Alan Muniz, Arturo Fern\'andez and the referee for useful suggestions which improved the readability of this paper.}


\end{document}